\documentclass[a4paper,12pt]{article} 
\usepackage{natbib}
\usepackage[utf8]{inputenc}
\usepackage[T1]{fontenc}
\usepackage{lmodern}
\usepackage{amsmath} 
\usepackage{amsthm}
\usepackage{amssymb}
\usepackage{mathtools}
\usepackage{bm}
\usepackage{ifthen}
\usepackage{overpic}
\usepackage{fancyhdr}
\usepackage{graphicx} 
\usepackage[left=1in,top=1in,bottom=1.5in,right=1in]{geometry}

\renewcommand{\subsectionmark}[1]{} 
\newcommand{\mc}[1]{\mathcal{#1}}

\newcommand{\eps}{\varepsilon}

\DeclareMathOperator{\re}{Re}
\DeclareMathOperator{\im}{Im}
\DeclareMathOperator{\Span}{span}
\DeclareMathOperator{\diag}{diag}

\newcommand*{\C}{{\mathbb{C}}}     
\newcommand*{\R}{{\mathbb{R}}}     
     
\newcommand*{\N}{{\mathbb{N}}}

\newcommand*{\Lin}{{\mathcal{L}}}   
\newcommand*{\Dom}{{\mathcal{D}}}   
\newcommand{\ran}{{\mathcal{R}}}   
\renewcommand{\ker}{{\mathcal{N}}}

\newcommand*{\abs} [1]{\lvert#1\rvert}
\newcommand*{\norm}[1]{\lVert#1\rVert}
\newcommand*{\set} [1]{\{#1\}}
\newcommand*{\setm}[2]{\{\,#1\mid#2\,\}}   
\newcommand*{\iprod}[2]{\langle#1,#2\rangle}

\newcommand*{\Setm}[2]{\left\{\,#1\,\middle|\,#2\,\right\}}
\newcommand*{\Lp}[1][p]{L^{#1}}

\newcommand{\pmat}[1]{\begin{pmatrix}#1\end{pmatrix}}
\newcommand{\pmatsmall}[1]{\begin{psmallmatrix}#1\end{psmallmatrix}}

\newcommand*{\Abs}[2][default]{\ifthenelse{\equal{#1}{default}}{\left\lvert#2\right\rvert}{\ldelim{#1}{\lvert}#2\rdelim{#1}{\rvert}}}
\newcommand*{\Norm}[2][default]{\ifthenelse{\equal{#1}{default}}{\left\lVert#2\right\rVert}{\ldelim{#1}{\lVert}#2\rdelim{#1}{\rVert}}}

\newcommand*{\Iprod}[3][default]{\ifthenelse{\equal{#1}{default}}{\left\langle#2,#3\right\rangle}{\ldelim{#1}{\langle}#2,#3\rdelim{#1}{\rangle}}}
\newcommand*{\Dualpair}[3][default]{\ifthenelse{\equal{#1}{default}}{\left\langle#2,#3\right\rangle}{\ldelim{#1}{\langle}#2,#3\rdelim{#1}{\rangle}}}

\newcommand*{\List}[2][1]{\set{#1,\ldots,#2}}

\newcommand{\eq}[1]{\begin{align*}#1\end{align*}}
\newcommand{\eqn}[1]{\begin{align}#1\end{align}}

\newcommand{\CL}{C_\Lambda }
\newcommand{\CeL}{C_{e\Lambda} }

\newcommand{\KL}{K_\Lambda }
\newcommand{\KtL}{K_2^\Lambda }
\newcommand{\KtoL}{K_{21}^\Lambda }

\newcommand{\KtotL}{K_2^\Lambda }

\newcommand{\Sylspace}{\Dom(\CL)\cap \Dom(\KtoL) }

\newcommand{\JY}[1][k]{J_{#1}^Y}

\newcommand{\ZI}{Z_0}
\newcommand{\sysops}{(\tilde{A},\tilde{B},\tilde{C},\tilde{D},\tilde{E},\tilde{F})}

\newcommand{\gs}{\sigma}
\newcommand{\ga}{\alpha}

\newcommand{\gl}{\lambda}
\newcommand{\gw}{\omega}

\newcommand*{\conj}[1]{\overline{#1}}

\newcommand{\ieq}[1]{$#1$}

\newcommand{\inv}{^{-1}}

\newcommand*{\ddb}[2][1]{\ifthenelse{\equal{#1}{1}}{\frac{d}{d#2}}{\frac{d^{#1}}{d#2^{#1}}}}
\newcommand*{\pd}[3][1]{\ifthenelse{\equal{#1}{1}}{\frac{\partial{#2}}{\partial{#3}}}{\frac{\partial^{#1}{#2}}{\partial#3^{#1}}}}

\newcommand{\Gconds}{$\mc{G}$-con\-di\-tions}
\newcommand{\yref}{y_{\mbox{\scriptsize\textit{ref}}}}
\newcommand{\Ops}{\mc{O}}
\newcommand*{\pinv}{^{\dagger}}

\newcommand*{\keyterm}[1]{\emph{#1}}

\newtheorem{theorem}{Theorem}[section]
\newtheorem{lemma}[theorem]{Lemma}

\newtheorem{assumption}[theorem]{Assumption}

\theoremstyle{definition}
\newtheorem{definition}[theorem]{Definition} 
\newtheorem{remark}[theorem]{Remark} 

\newenvironment{RORP}{\medskip\noindent\textbf{The Robust Output Regulation Problem.}}{}

\bibpunct{[}{]}{,}{n}{}{,}

\title{Controller Design for Robust Output Regulation of Regular Linear Systems}
\author{Lassi Paunonen\thanks{Tampere University of Technology, PO.Box 553, 33101 Tampere, Finland, \texttt{lassi.paunonen@tut.fi}}}
\date{~}

\begin{document}

\maketitle
\vspace{-8ex}

\thispagestyle{plain}

\begin{abstract}
We present three dynamic error feedback controllers for robust output regulation of regular linear systems. These controllers are (i) a minimal order robust controller for exponentially stable systems (ii) an observer-based robust controller and (iii) a new internal model based robust controller structure. In addition, we present two controllers that are by construction robust with respect to predefined classes of perturbations. The results are illustrated with an example where we study robust output tracking of a sinusoidal reference signal for a two-dimensional heat equation with boundary control and observation.  
\end{abstract}


\section{Introduction}

The topic of this paper is the construction of controllers for robust output regulation of linear infinite-dimensional systems. The goal in this control problem is to design a control law for a linear system
\begin{subequations}
  \label{eq:plantintro}
  \eqn{
  \dot{x}(t)&= Ax(t)+Bu(t) + w(t), \qquad x(0)=x_0\in X\\
  y(t)& = Cx(t) + Du(t)
  \label{eq:plantintroout}
  }
\end{subequations}
in such a way that the output $y(t)$ converges asymptotically to a given reference signal $\yref(t)$ despite the external disturbance signal $w(t)$. 
In addition, the controller must tolerate small perturbations and uncertainties in the parameters $(A,B,C,D)$ of the plant~\eqref{eq:plantintro}. The robust output regulation problem was first studied for finite-dimensional systems in the 1970's most notably by Francis and Wonham~\cite{FraWon75a,FraWon76}, and Davison~\cite{Dav76}, and since then the theory of output regulation has been been actively developed for infinite-dimensional systems~\cite{Poh82,Sch83b,ByrLau00,HamPoh00,RebWei03,Imm07a}.

The most recent developments in the field 
are related to the study of output regulation and robust output regulation for infinite-dimensional systems with unbounded input and output operators, and
especially for \keyterm{regular linear systems}~\cite{Wei89b,Wei94,StaWei02} which are often encountered in the study of partial differential equations with boundary control and observation~\cite{ByrGil02}. 
In particular, the characterization of the solvability of the output regulation problem using the so-called \keyterm{regulator equations} was extended for systems with unbounded operators $B$ and $C$ in~\cite{PauPohCDC13,NatGil14}, and the \keyterm{internal model principle} of robust output regulation was generalized for regular linear systems in~\cite{PauPoh14a}.

In this paper we continue the work begun in~\cite{PauPoh14a}.  
The main emphasis in the reference~\cite{PauPoh14a} was on studying the properties of robust controllers and on characterizing the solvability of the robust output regulation problem. In this paper we concentrate on 
designing actual controllers that achieve robust output regulation for the regular linear system~\eqref{eq:plantintro}.
As our main results we present three 
different robust controllers.
Two of
these
controllers employ structures that are familiar from the control of systems with bounded operators $B$ and $C$, 
and 
the third
employs a completely new complementary controller structure.

The reference signal $\yref(\cdot)$ and the disturbance signal $w(\cdot)$ are assumed to be generated by an \keyterm{exosystem}%
\begin{subequations}
  \label{eq:exointro}
  \eqn{
  \dot{v}(t)&=Sv(t), \qquad v(0)=v_0\in W\\
  w(t)&=Ev(t)\\
  \yref(t)&=-Fv(t)
  }
\end{subequations}
on a finite-dimensional space $W=\C^r$. Here $S$ is a matrix with eigenvalues $\gs(S)=\set{i\gw_1,\ldots,i\gw_q}\subset i\R$.
The main objective in this paper is to
achieve robust output regulation for the
system~\eqref{eq:plantintro} by choosing appropriate 
parameters $(\mc{G}_1,\mc{G}_2,K)$ for the dynamic error feedback controller
\begin{subequations}
  \label{eq:controllerintro}
  \eqn{
  \dot{z}(t) &= \mc{G}_1z(t)+\mc{G}_2e(t), \qquad z(0)=z_0\in Z\\
  u(t) &= Kz(t).
  }
\end{subequations}
where $e(t)=y(t)-\yref(t)$ is the \keyterm{regulation error}.

The main tool in
constructing robust controllers
is the internal model principle, which provides a complete characterization of the controllers that achieve robust output regulation for the system~\eqref{eq:plantintro} and for the reference and disturbance signals generated by the exosystem~\eqref{eq:exointro}.
In particular, this fundamental result tells us that
the control problem
can be solved by including a suitable \keyterm{internal model} of the dynamics of the exosystem into the controller~\eqref{eq:controllerintro}, and by choosing the rest of the parameters of the controller in such a way that the closed-loop system consisting of the plant and the controller is stable. The classical definition of the internal model (also referred to as the \keyterm{p-copy internal model}) requires that if 
$p$
is the dimension of the output space $Y$ and  
if~$S$ has a Jordan block of dimension $n_k$ associated to an eigenvalue $i\gw_k$, then the operator~$\mc{G}_1$ must have at least $p$ independent Jordan chains of length greater or equal to $n_k$ associated to the same eigenvalue $i\gw_k$~\cite{FraWon75a,PauPoh10}. 
In this paper 
we also use an alternative definition for an internal model, 
called
the \keyterm{\Gconds}, which is applicable 
even if
$Y$ is infinite-dimensional~\cite{HamPoh10,PauPoh14a}.

The first controller in this paper presented in Section~\ref{sec:Mincontr} is constructed by choosing $\mc{G}_1$ as the internal model of the exosystem~\eqref{eq:exointro} and by stabilizing the closed-loop system with suitable choices of $\mc{G}_2$ and $K$. It is well-known that if the plant~\eqref{eq:plantintro} is exponentially stable and $S$ is diagonal, then this very simple structure is extremely effective. Indeed, this controller
structure has been successfully used on several occasions for infinite-dimensional systems with bounded and unbounded input and output operators~\cite{UkaIwa90,HamPoh96a,LogTow97,HamPoh00,RebWei03}.
The most important advantages of this controller structure is that due to the internal model principle, this controller is of minimal possible order, and if $\dim Y<\infty$ then the resulting controller is
finite-dimensional. 
In~\cite{Sai15phd} this type of structure was used for regular linear systems on Hilbert spaces and with $U=Y$. In this paper we present a minimal order controller that solves the robust output regulation for a regular linear system~\eqref{eq:plantintro} on a Banach space $X$, without restrictions on the input and output spaces, and with the most general choices for the stabilizing operators $\mc{G}_2$ and $K$.

In Section~\ref{sec:Mincontr} we in addition present a separate version of the minimal order controller for a situation where robustness is only required with respect to a predefined class $\Ops_0$ of perturbations.
The design is motivated by the recent observation~\cite{PauPoh13a,Pau15a} that in such a situation the
robust output regulation problem
may be solvable 
with a controller incorporating a \keyterm{reduced order internal model} that is strictly smaller than the full p-copy internal model.
In particular, in~\cite{Pau15a} such a controller was successfully designed for a given class $\Ops_0$ of perturbations. In this paper we present a new controller that solves the robust output regulation problem for a stable regular linear system and for a given class $\Ops_0$.
This new controller has the advantage over the one presented in~\cite{Pau15a} in that the controller is of minimal order, and it is finite-dimensional whenever $\dim Y<\infty$. This controller is new even for finite-dimensional linear systems and for infinite-dimensional systems with bounded operators $B$ and $C$.

The second robust controller of this paper presented in Section~\ref{sec:ContrOne} employs a novel structure that was introduced in~\cite{Pau15a} for construction of controllers with reduced order internal models.
In particular, the system operator $\mc{G}_1$  of the controller has a triangular structure that is naturally complementary to 
the structure of observer-based robust controllers~\cite{Dav76,HamPoh10}. The main advantages of this new controller are that it has the natural structure for the inclusion of the p-copy internal model into the dynamics of the controller, and that it can robustly regulate plants that have a larger number of inputs than outputs. The construction of this second controller is a new result even for finite-dimensional linear systems and for infinite-dimensional systems with bounded operators~$B$ and $C$. In Section~\ref{sec:ContrOne} we also use the same structure to generalize the original reduced order internal model based controller in~\cite{Pau15a} for regular linear systems.

Finally, the third robust controller
presented in Section~\ref{sec:ContrTwo} is an observer-based controller that employs the triangular structure that was successfully used
for robust output regulation of
systems with bounded~$B$ and~$C$ together with infinite-dimensional diagonal exosystems in~\cite{HamPoh10}. In this paper we generalize the observer-based construction in~\cite{HamPoh10} 
to regular linear systems with nondiagonal exosystems.

As the first main result in this paper we present the internal model principle. This result was first generalized for regular linear systems in~\cite{PauPoh14a} in the more general setting of infinite-dimensional exosystems and 
strong stability of the closed-loop system.
In this paper we introduce it for regular linear systems with finite-dimensional exosystems and exponential closed-loop stability. 
We demonstrate that the exponential closed-loop stability allows simplifying general assumptions of the theorem, and show that the regulation error has an exponential rate of decay.  

We illustrate the construction of controllers 
by considering the robust output regulation problem for a two-dimensional heat equation with boundary control and observation. We begin by stabilizing the system with negative output feedback, and we subsequently construct a minimal order 
controller that achieves 
robust
tracking of a sinusoidal reference signal. 

The paper is organized as follows. The standing assumptions on the plant, the exosystem and the controller are stated in Section~\ref{sec:plantexo}. In Section~\ref{sec:RORP} we formulate the robust output regulation problem and present the internal model principle. 
The minimal order controller for stable systems is presented in Section~\ref{sec:Mincontr}. 
In Sections~\ref{sec:ContrOne} we introduce the new controller structure for robust output regulation. Finally, the observer-based robust controller is constructed in~\ref{sec:ContrTwo}. The robust output tracking of the two-dimensional heat equation is considered in~\ref{sec:heatex}.

\section{The Plant, The Exosystem and The Controller}
\label{sec:plantexo}

If $X$ and $Z$ are Banach spaces and $A:X\rightarrow Z$ is a linear operator, we denote by $\Dom(A)$, $\ker(A)$ and $\ran(A)$ the domain, kernel and range of $A$, respectively. The space of bounded linear operators from $X$ to $Z$ is denoted by $\Lin(X,Z)$. If \mbox{$A:X\rightarrow X$,} then $\gs(A)$, $\gs_p(A)$ and $\rho(A)$ denote the spectrum, the point spectrum and the \mbox{resolvent} set of $A$, respectively. For $\gl\in\rho(A)$ the resolvent operator is given by \mbox{$R(\gl,A)=(\gl -A)^{-1}$}. The inner product on a Hilbert space is denoted by $\iprod{\cdot}{\cdot}$. For an operator $A: \Dom(A)\subset X\to Z_1\times\cdots \times Z_n$ we use the notation $A = (A_k)_{k=1}^n $, where $A_k : \Dom(A)\subset X\to Z_k$ for all $k\in \List{n}$, to signify that $Ax = (A_k x)_{k=1}^n$ for $x\in \Dom(A)$. On the other hand, for an operator $A\in \Lin(X_1\times \cdots\times X_n,Z)$ we use the notation $A=(A_1,\ldots,A_n)$, meaning that $Ax = \sum_{k=1}^n A_kx_k$ for all $x=(x_k)_{k=1}^n\in X_1\times \cdots\times X_n$.

We consider a linear system~\eqref{eq:plantintro} on a Banach space $X$ with state $x(t)\in X$, output $y(t)\in Y$, and input $u(t)\in U$. The spaces $U$ and $Y$ are Hilbert spaces.
The operator $A:\Dom(A)\subset X\rightarrow X$ generates a strongly continuous semigroup $T(t)$ on $X$. For a fixed
$\gl_0\in\rho(A)$
we define the scale spaces $X_1 = (\Dom(A), \norm{(\gl_0-A)\cdot})$ and $X_{-1}= \overline{(X,\norm{R(\gl_0,A)\cdot})}$ (the completion of $X$ with respect to the norm $\norm{R(\gl_0,A)\cdot}$)~\cite[Sec. II.5]{EngNag00book}.
The extension of $A$ to $X_{-1}$ is denoted by $A_{-1}: X\subset X_{-1}\rightarrow X_{-1}$.

Throughout the paper we assume that~\eqref{eq:plantintro} is a \keyterm{regular linear system}~\cite{Wei89b,Wei94,StaWei02,Sta05book}.
In particular,  $B\in \Lin(U,X_{-1})$ and $C\in \Lin(X_1,Y)$ are admissible with respect to $A$ and $D\in \Lin(U,Y)$. 
The operator $C$ 
in~\eqref{eq:plantintroout}
is replaced
with its 
$\Lambda$-extension
\eq{
\CL x = \lim_{\gl\to \infty}\, \gl C R(\gl,A)x
}
with $\Dom(\CL)$ consisting of those $x\in X$ for which the limit exists. If $C\in\Lin(X,Y)$, then $\CL=C$. 
For a regular linear system we have $\ran(R(\gl,A)B)\subset \Dom(\CL)$ for all $\gl\in\rho(A)$, and the transfer function of~\eqref{eq:plantintro} is~\cite[Sec. 4]{StaWei02}
\eq{
P(\gl) = \CL R(\gl,A)B + D \qquad \forall \gl\in\rho(A).
}
Finally, we define $X_B=\Dom(A) + \ran(R(\gl_0,A_{-1})B)\subset \Dom(\CL)$, which is independent of the choice of $\gl_0\in\rho(A)$.

\begin{assumption}
  \label{ass:stabdetect}
The pair $(A,B)$ is exponentially stabilizable 
and there exists $L\in\Lin(Y,X)$ such that $A+L\CL$ generates an exponentially stable semigroup.
  \end{assumption}
  The stabilizability of $(A,B)$ means that there exists
  $K\in \Lin(X_1,U)$ such that
  $(A,B,\KL)$ is a regular linear system for which $I$ is an admissible feedback operator, and $(A+B\KL)\vert_X$ generates an exponentially stable semigroup~\cite{WeiCur97}.

The exosystem~\eqref{eq:exointro} is a linear system on the finite-dimensional space 
$W=\C^r$ for some $r\in \N$, and $S\in \Lin(W)=\C^{r\times r}$,
 $E\in \Lin(W,X)$, and $F\in \Lin(W,Y)$.
  We assume the geometric multiplicity of each of the eigenvalues $\gs(S)= \set{i\gw_k}_{k=1}^q\subset i\R$ is equal to one.
We denote by $n_k\in\N$ the size of the Jordan block associated to $i\gw_k\in\gs(S)$. 
The following standing assumption is crucial for the solvability of the robust output regulation problem.
An immediate consequence of this assumption is that in order to achieve robust output regulation it is necessary that $\dim U\geq \dim Y$.

\begin{assumption}
For every $k\in \List{q}$ we have
  $i\gw_k\in\rho(A)$ and $P(i\gw_k)\in \Lin(U,Y)$ is surjective.
\end{assumption}

The dynamic error feedback controller~\eqref{eq:controllerintro} is 
an abstract linear system on a Banach space $Z$. The operator $\mc{G}_1:\Dom(\mc{G}_1)\subset Z\rightarrow Z$ generates a semigroup on $Z$, and $\mc{G}_2\in\Lin(Y,Z)$ and $K\in\Lin(Z_1,U)$ is admissible with respect to $\mc{G}_1$. 
The operator $K$ in~\eqref{eq:controllerintro} is replaced with its $\Lambda$-extension $\KL$.

The closed-loop system consisting of the plant~\eqref{eq:plantintro} and the controller~\eqref{eq:controllerintro} on the Banach space $X_e=X\times Z$ with state $x_e(t)=(x(t), z(t))^T$ is of the form 
  \eq{
  \dot{x}_e(t) &= A_ex_e(t)+B_ev(t), \qquad x_e(0)=x_{e0}\in X_e ,\\
  e(t) &=\CeL x_e(t)+D_ev(t),
  }
where $e(t)= y(t)- \yref(t)$ is the 
regulation error,
$x_{e0}=(x_0, z_0)^T$,  $C_e= (\CL, \; D\KL)$, 
$D_e=F$,
  \eq{
  A_e=\pmat{A_{-1}&B\KL\\\mc{G}_2\CL&\mc{G}_1+\mc{G}_2D\KL}, \qquad 
  B_e=\pmat{E\\\mc{G}_2F}.
  }
The operator $A_e: \Dom(A_e)\subset X_e\to X_e$ has the
domain
\eq{
\Dom(A_e)
&=  \Setm{(x,z)\in X_B\times \Dom(\mc{G}_1)}{A_{-1}x+B\KL z\in X}
}
where $X_B = \Dom(A) + \ran(R(\gl_0,A_{-1})B)$, and
 $\Dom(C_e) = \Dom(\CL)\times \Dom(\KL)\supset \Dom(A_e)$, $B_e \in \Lin(W,X\times Z)$ and $D_e\in\Lin(W,Y)$. Here $\CeL$ is the $\Lambda$-extension of $C_e$.

  \begin{theorem}
    \label{thm:CLreg}
The closed-loop system $(A_e,B_e,C_e,D_e)$ is a regular linear system.
    \end{theorem}

    \begin{proof}
      See~\cite[Sec. 8]{PauPoh14a}.
    \end{proof}

\section[The Robust Output Regulation Problem]{The Robust Output Regulation Problem and the Internal Model Principle}
\label{sec:RORP}

We can now formulate the robust output regulation problem. 
We consider perturbations $(\tilde{A},\tilde{B},\tilde{C},\tilde{D},\tilde{E},\tilde{F})\in \Ops$ of 
$(A,B,C,D,E,F)$ where the operators in the class $\Ops$ of admissible perturbations are such that (i) the perturbed plant $(\tilde{A},\tilde{B},\tilde{C},\tilde{D})$ is a regular linear system and (ii) $i\gw_k\in \rho(\tilde{A})$ for all $k\in \List{q}$. 
These two conditions are in particular satisfied for all bounded and sufficiently small perturbations to $(A,B,C,D)$, and for arbitrary bounded perturbations to the operators $E$ and $F$.

\begin{RORP}
  Choose the controller $(\mc{G}_1,\mc{G}_2,K)$ in such a way that the following are satisfied:
\begin{itemize}
  \item[\textup{(a)}] The closed-loop semigroup $T_e(t)$ is exponentially stable.
  \item[\textup{(b)}] 
  For all initial states $x_{e0}\in X_e$ and $v_0\in W$ the regulation error satisfies $e^{\ga\cdot}e(\cdot)\in\Lp[2](0,\infty;Y)$ for some $\ga>0$.
\item[\textup{(c)}] If the operators $(A,B,C,D,E,F)$ are perturbed to 
  $(\tilde{A},\tilde{B},\tilde{C},\tilde{D},\tilde{E},\tilde{F})\in \Ops$ in such a way that the closed-loop system remains exponentially stable, then for all initial states $x_{e0}\in X_e$ and  $v_0\in W$ the regulation error satisfies $e^{\tilde{\ga}\cdot}e(\cdot)\in\Lp[2](0,\infty;Y)$ for some $\tilde{\ga}>0$.  \end{itemize}
\end{RORP}

We have from~\cite[Sec. 4]{PauPoh14a} that for initial states $x_{e0}\in \Dom(A_e)$ the regulation error $e(\cdot)$ is a continuous function and $\lim_{t\to\infty}e(t)=0$ whenever the property (b) holds. Thus for such initial states the condition $e^{\ga\cdot}e(\cdot)\in\Lp[2](0,\infty;Y)$ for an $\ga>0$ implies that the regulation error decays to zero at an exponential rate.

In the following we present two definitions for an internal model~\cite{PauPoh10,PauPoh14a}.
In Definition~\ref{def:pcopy}
``independent Jordan chains'' refer
to chains originating from linearly independent eigenvectors of $\mc{G}_1$.

\begin{definition}
  \label{def:pcopy}
  Assume $\dim Y<\infty$. A controller $(\mc{G}_1,\mc{G}_2,K)$ is said to\/ {\em incorporate a p-copy internal model} of the exosystem $S$ if for all $k\in\List{q}$ we 
have
\eq{
\dim\ker(i\gw_k-\mc{G}_1)\geq\dim\, Y
}
and\/ $\mc{G}_1$ has at least\/ $\dim Y$ independent Jordan chains of length greater than or equal to~$n_k$ associated to 
the eigenvalue
$i\gw_k$.
\end{definition}

\begin{definition}
  \label{def:Gconds}
  A controller\/ $(\mc{G}_1,\mc{G}_2,K)$ is said to satisfy the\/ {\em\Gconds} 
  if
\begin{subequations}
  \label{eq:Gconds}
\eqn{
\ran(i\gw_k-\mc{G}_1)\cap\ran(\mc{G}_2)&=\set{0} ~\qquad ~ \forall k\in\List{q} , \label{eq:Gconds1} 
\\[.3ex]
\ker(\mc{G}_2)&=\set{0}, \label{eq:Gconds2}\\
\label{eq:Gconds3}
\MoveEqLeft[9] \ker(i\gw_k-\mc{G}_1)^{n_k-1} 
\subset\ran(i\gw_k-\mc{G}_1) \quad  \forall k\in\List{q}.
}
\end{subequations} 
\end{definition}

The following lemma gives a sufficient condition for invariance of the \Gconds\ in the situation where the matrix $S$ of the exosystem is diagonal. 

\begin{lemma}
  \label{lem:Gcondsinvariance}
  Let $S$ be a diagonal matrix.
  If the operators $(\mc{G}_1,\mc{G}_2)$ satisfy   the \Gconds,
  and if $K: \Dom(\mc{G}_1)\subset Z\to Y$ is such that
  $\ker(i\gw_k-\mc{G}_1)\subset \ker(K)$ for all $k\in\List{q}$, then also $(\mc{G}_1+ \mc{G}_2K,\mc{G}_2)$ satisfy   the \Gconds.
\end{lemma}

\begin{proof}
  Since $S$ is a diagonal matrix, we have $n_k=1$ for all $k\in \List{q}$ and the condition~\eqref{eq:Gconds3} is trivially satisfied.
  Because the condition $\ker(\mc{G}_2)=\set{0}$ is identical for both $(\mc{G}_1 + \mc{G}_2K,\mc{G}_2)$ and $(\mc{G}_1 ,\mc{G}_2)$, it is sufficient to show that $\ran(i\gw_k-\mc{G}_1+\mc{G}_2K)\cap\ran(\mc{G}_2)=\set{0}$ for all $k$.
  To this end, let $w=(i\gw_k-\mc{G}_1-\mc{G}_2K)z=\mc{G}_2y$ for some $k\in\List{q}$, $z\in \Dom(\mc{G}_1)$ and $y\in Y$. This implies $(i\gw_k-\mc{G}_1)z=\mc{G}_2(y+Kz)\in \ran(i\gw_k-\mc{G}_1)\cap \ran(\mc{G}_2)$, and we thus have $z\in \ker(i\gw_k-\mc{G}_1)$. Due to our assumptions we then also have $Kz=0$ and $w=(i\gw_k-\mc{G}_1)z=\mc{G}_2y$, which finally  imply $w=0$ due to~\eqref{eq:Gconds1}.
\end{proof}

The following theorem presents the internal model principle for regular linear systems with finite-dimensional exosystems and exponential closed-loop stability.

\begin{theorem}
  \label{thm:RORPchar}
  Assume that the controller stabilizes the closed-loop system exponentially.
Then the controller solves the robust output regulation problem if and only if it satisfies the \Gconds. 

Moreover, if $\dim Y<\infty$, then the controller solves the robust output regulation problem if and only if it incorporates a p-copy internal model of the exosystem.
\end{theorem}

\begin{proof}
Since $A_e$ generates an exponentially stable semigroup and $S$ is a matrix with spectrum on $i\R$, the Sylvester equation $\Sigma S = A_e\Sigma + B_e$ has a unique solution $\Sigma\in \Lin(W,X_e)$ satisfying $\ran(\Sigma)\subset \Dom(A_e)$~\cite{Vu91}.
Because an exponentially stable semigroup is also strongly stable, and since $i\R\subset \rho(A_e)$, we have from~\cite[Thm. 7.2]{PauPoh14a} that the controller satisfies the \Gconds\ if and only if it solves the robust output regulation problem as defined in the reference~\cite{PauPoh14a}. 
The definition of the robust output regulation problem in~\cite{PauPoh14a} can be obtained from our problem statement with the following modifications: 
\begin{itemize}
  \item[(i)] The exponential closed-loop stability is replaced by strong stability.
  \item[(ii)] 
    It is assumed that
    for all admissible perturbations 
    the Sylvester equation $\Sigma S=\tilde{A}_e\Sigma +\tilde{B}_e$ has a solution.
  \item[(iii)] The condition $e^{\ga\cdot}e(\cdot)\in \Lp[2](0,\infty;Y)$ for $x_{e0}\in X_e$ is replaced by $\lim_{t\to\infty}e(\cdot)=0$ for $x_{e0}\in \Dom(A_e)$.
\end{itemize}

We begin by showing that under the assumption of exponential closed-loop stability the two conditions in~(iii) are equivalent.
We prove this only for the nominal closed-loop system $(A_e,B_e,C_e,D_e)$.  For perturbed parameters
the situation can be handled analogously.
We have from~\cite[Lem. 4.3]{PauPoh14a} that 
  \eq{
  x_e(t) = T_e(t)x_{e0}-T_e(t)\Sigma v_0 + \Sigma T_S(t)v_0
  }
  for all $x_{e0}\in X_e$ and $v_0\in W$, and since $(A_e,B_e,C_e,D_e)$ is a regular linear system,
  \eq{
 e(t) =C_{e\Lambda}T_e(t)(x_{e0}-\Sigma v_0) + (C_e\Sigma + D_e)T_S(t)v_0
  }
  is defined for almost all $t\geq 0$. In addition, if $x_{e0}\in \Dom(A_e)$, then $e(t)$ is continuous and is given by the above formula for all $t\geq 0$. 
  The error contains the two terms $e(t)=e_1(t) + e_2(t)$. The second term $e_2(\cdot) = (C_e\Sigma + D_e)T_S(\cdot)v_0$ is continuous and it is either nonvanishing or identically zero~\cite[Lem. A.1]{PauPoh14a}.
  Since $T_e(t)$ is exponentially stable, for some $\ga>0$ the first term satisfies $e^{\ga\cdot}e_1(\cdot)\in \Lp[2](0,\infty;Y)$ for all $x_{e0}\in X_e$ and $v_0\in W$. These properties imply, under the assumption of exponential stability of the closed-loop system, that the regulation error satisfies $e^{\ga\cdot}e(\cdot)\in \Lp[2](0,\infty;Y)$ for all $x_{e0}\in X_e$ and $v_0\in W$ if and only if $\lim_{t\to\infty}e(t)=0$ for all $x_{e0}\in \Dom(A_e)$ and $v_0\in W$.
Thus the conditions in~(iii) are equivalent.

Assume now that the controller satisfies the \Gconds.
The class of admissible perturbations in this paper is strictly smaller than the class of perturbations in~\cite{PauPoh14a} because exponential stability is stronger than strong stability, and because 
$\Sigma S=\tilde{A}_e\Sigma +\tilde{B}_e$ has a solution for any perturbations for which the closed-loop system is exponentially stable~\cite{Vu91}. Because of this, and 
because we assumed the exponential closed-loop stability, we have from~\cite[Thm. 7.2]{PauPoh14a} that the controller satisfying the \Gconds\ solves the robust output regulation problem as defined in this paper.

 Conversely, we can now assume that the controller solves the robust output regulation problem. It then follows from the proof of~\cite[Thm. 7.2]{PauPoh14a} that the controller must satisfy the \Gconds\ provided that the class of admissible perturbations contains $\tilde{E}=0$ (corresponding to the zero disturbance signal) and arbitrary bounded
perturbations to the operator~$F$ of the exosystem. Because these perturbations do not affect the stability of the closed-loop system, they also belong to the class $\Ops$ of perturbations in this paper. This concludes that the controller indeed satisfies the \Gconds.

Finally, if $\dim Y<\infty$, we similarly have from~\cite[Thm. 6.2]{PauPoh14a} that the controller solves the robust output regulation problem if and only if it incorporates a p-copy internal model of the exosystem.
\end{proof}

\section{The Minimal Robust Controller for Stable Systems}
\label{sec:Mincontr}

In this section we construct a minimal order robust controller under the assumption that the system operator $A$ of the regular linear system~\eqref{eq:plantintro} generates an exponentially stable semigroup
and the matrix $S$ of the exosystem is diagonal, i.e.,
\eq{
S=\diag(i\gw_1,i\gw_2,\ldots,i\gw_q)\in \C^{q\times q}.
} 
We begin by choosing the parameters of the controller.
In this controller structure the system operator $\mc{G}_1$ contains precisely the internal model of the exosystem~\eqref{eq:exointro}. This is achieved by defining
$Z = Y^q $, and
\eq{
\mc{G}_1 &= \diag \bigl( i\gw_1 I_Y,\ldots,i\gw_q I_Y \bigr)
\in \Lin(Z),
\qquad K = \eps K_0 = \eps\bigl(K_0^1,\ldots,K_0^q\bigr)
\in \Lin(Z,U)
}
where $\eps>0$ and $K_0\in \Lin(Z,U)$. We choose the components $K_0^k\in \Lin(Y,U)$ of $K_0$ in such a way that the operators $P(i\gw_k)K_0^k$ are invertible. This is possible  due to the assumption of surjectivity of $P(i\gw_k)$, and can be achieved, for example, by choosing $K_0^k=P(i\gw_k)\pinv$
(the Moore--Penrose pseudoinverse of $P(i\gw_k)$)
for all $k\in\List{q}$. Finally, we choose 
\eq{
\mc{G}_2
=(\mc{G}_2^k)_{k=1}^q
= (-(P(i\gw_k)K_0^k)^\ast)_{k=1}^q
\in \Lin(Y,Z).
}
If we make the choice $K_0^k=P(i\gw_k)\pinv$, then $\mc{G}_2^k=-I_Y$ for all $k\in \List{q}$.

\begin{theorem}
  \label{thm:SimpleContrMain}
Assume that the semigroup $T(t)$ generated by $A$ is exponentially stable and $S$ is a diagonal matrix.
  Then there exists $\eps^\ast>0$ such that for any $0<\eps\leq \eps^\ast$ the controller with the above choices of parameters solves the robust output regulation problem.

  In particular, the operators $(\mc{G}_1,\mc{G}_2)$ satisfy the \Gconds\ and the closed-loop system is exponentially stable for all $0<\eps\leq \eps^\ast$.
\end{theorem}

\begin{proof}
We begin by showing that the controller satisfies the \Gconds. 
Since $K_0^k$ were chosen in such a way that $P(i\gw_k)K_0^k$ are invertible for all $k\in \List{q}$, we have that $\ker(\mc{G}_2)=\set{0}$. Let 
$k\in \List{q}$, 
$z,z_1\in Z$, and $y\in Y$ be such that $z=(i\gw_k-\mc{G}_1)z_1=\mc{G}_2y$. The diagonal structure of $\mc{G}_1$ implies that we then necessarily have $0=\mc{G}_2^k y =-(P(i\gw_k)K_0^k)^\ast y$, which is only possible if $y=0$ since $P(i\gw_k)K_0^k$ is invertible. This further implies $z=\mc{G}_2y=0$. Since $k\in \List{q}$ and $z\in \ran(i\gw_k-\mc{G}_1)\cap \ker(\mc{G}_2)$ were arbitrary, this concludes $\ran(i\gw_k-\mc{G}_1)\cap \ran(\mc{G}_2)=\set{0}$.
Finally, since $n_k=1$ for all $k\in \List{q}$, the condition~\eqref{eq:Gconds3} is trivially satisfied.

We define  $H = (H_1,H_2,\ldots,H_q)\in \Lin(Z,X)$ 
by choosing 
\eq{
H_k =  R(i\gw_k,A_{-1}) B K_0^k
}
for all $k\in \List{q}$.
Due to the diagonal structure of $\mc{G}_1$, it is easy to see that this operator is the unique solution of the Sylvester equation $H \mc{G}_1 = A_{-1}H+ BK_0$.
Clearly $\ran(H)\subset X_B$ and we can define $C_0=\CL H+DK_0\in \Lin(Z,Y)$. 
The operator $C_0$ is of the form $C_0=(C_0^1,\ldots,C_0^q)$. A direct computation shows that
\eq{
C_0^k  
= \CL H_k +DK_0^k
= \CL R(i\gw_k,A_{-1})BK_0^k +DK_0^k
= P(i\gw_k)K_0^k,
}
and thus $C_0 = -\mc{G}_2^\ast$.

It remains to show that there exists $\eps^\ast>0$ such that the closed-loop system is exponentially stable for all $0<\eps\leq \eps^\ast$. The closed-loop system operator is given by
\eq{
A_e  = \pmat{A_{-1}&\eps BK_0\\\mc{G}_2\CL & \mc{G}_1 + \eps \mc{G}_2 DK_0},
\qquad \Dom(A_e)  = \Setm{(x,z)\in X_B\times Z}{A_{-1}x+\eps BK_0z\in X}.
}
If we choose a similarity transformation 
\eq{
Q_e = \pmat{-I&\eps H\\0&I} = Q_e\inv\in \Lin(X\times Z)
} 
we can define $\hat{A}_e = Q_eA_eQ_e\inv$ with domain $\Dom(\hat{A}_e)=\Setm{x_e\in X_e}{Q_e\inv x_e \in \Dom(A_e)}$. Using $\ran(H)\subset X_B$ and $\ran(A_{-1}H+BK_0)=\ran(H \mc{G}_1)\subset X$ the 
condition $Q_e\inv x_e\in \Dom(A_e)$  for $x_e=(x,z)\in X\times Z$ 
becomes
\eq{
Q_e\inv x_e \in \Dom(A_e) 
\quad \Leftrightarrow\quad 
& \left\{
\begin{array}{l}
  -x + \eps Hz \in X_B\\
 - A_{-1}x + \eps A_{-1} Hz +\eps BK_0 z\in X\\
\end{array}
\right.
\\
\Leftrightarrow \quad  & ~
x_e\in \Dom(A)\times Z,
}
and thus $\Dom(\hat{A}_e)=\Dom(A)\times Z$.
Now for any $x_e = (x,z)\in \Dom(\hat{A}_e)$ a direct computation using $H \mc{G}_1=A_{-1}H+BK_0$ and $\CL H+DK_0=-\mc{G}_2^\ast$ shows that
  \eq{
  \MoveEqLeft[-.5]\hat{A}_e x_e 
  = Q_eA_e \pmat{-x+\eps Hz\\z}\\
  &=\pmat{(A-\eps H\mc{G}_2\CL)x +\eps (-A_{-1}H-BK_0+ H\mc{G}_1)z-\eps^2 H\mc{G}_2 \mc{G}_2^\ast z\\-\mc{G}_2\CL x+ (\mc{G}_1-\eps \mc{G}_2 \mc{G}_2^\ast)z} \\
  &=\pmat{(A-\eps H\mc{G}_2\CL)x -\eps^2 H\mc{G}_2 \mc{G}_2^\ast z\\-\mc{G}_2\CL x+ (\mc{G}_1-\eps \mc{G}_2 \mc{G}_2^\ast)z} \\
  &=\left[ \pmat{A-\eps H\mc{G}_2\CL &0\\-\mc{G}_2\CL &\hspace{-.3cm}\mc{G}_1-\eps \mc{G}_2 \mc{G}_2^\ast}+\eps^2 \pmat{0&\hspace{-.2cm}-H\mc{G}_2 \mc{G}_2^\ast\\0&0} \right] \pmat{x\\z}
  }
Since $C$ is admissible with respect to $A$, we have from the results in~\cite[Sec. III.3.c]{EngNag00book} that there exists $\eps_1>0$ such that $A+\eps H\mc{G}_2\CL $ generates an exponentially stable semigroup provided that $0<\eps\leq \eps_1$.
Moreover, Lemma~\ref{lem:G1fbstab} shows that the semigroup generated by $\mc{G}_1-\eps \mc{G}_2 \mc{G}_2^\ast$ is exponentially stable for all $\eps>0$, since $\sqrt{\eps}\mc{G}_2^k = -\sqrt{\eps} (P(i\gw_k)K_0^k)^\ast$ are invertible for all $k\in \List{q}$.  
Because $\CL$ is an admissible input operator for $A-\eps H \mc{G}_2 \CL$, $\mc{G}_2\in \Lin(Y,Z)$, and the diagonal operators generate exponentially stable semigroups, the semigroup generated by the triangular operator is exponentially stable for all $0<\eps\leq \eps_1$. Furthermore, because the second term is a bounded operator, it follows from standard perturbation theory of semigroups and similarity that there there exists $\eps^\ast>0$ such that $A_e$ is exponentially stable for all $0<\eps\leq \eps^\ast$.  

  Since the controller satisfies the \Gconds\ and the closed-loop system is exponentially stable for all $0<\eps\leq \eps^\ast$, we have from Theorem~\ref{thm:RORPchar} that for any $0<\eps\leq \eps^\ast$ the controller solves the robust output regulation problem.
\end{proof}

\begin{remark}
  \label{rem:stabbyoutfb}
  The controller presented in this section can also be used if the 
  plant is initially unstable but can be stabilized with output feedback,
  i.e., 
there exists
  an admissible feedback element $K_1\in \Lin(Y,U)$ 
such that
the semigroup generated by
  $(A+BK_1(I-DK_1)\inv \CL)\vert_X$ is exponentially stable.
  Indeed, in such a case 
  the controller can be designed for the stabilized system
  $( (A+BK_1(I-DK_1)\inv \CL)\vert_X,B(I-K_1D)\inv,(I-DK_1)\inv \CL,(I-DK_1)\inv D)$.
  This 
  procedure is demonstrated in Section~\ref{sec:heatex}.
\end{remark}

\begin{remark}
  \label{rem:realcontroller}
If the plant is \keyterm{real} in the sense that $P(-i\gw)=\conj{P(i\gw)}$ for all $\gw\in\R$, if $Y=\C^p$, $U=\C^m$, and if the exosystem is of the form
\eq{
S = \diag(i\gw_1,-i\gw_1,\ldots,i\gw_q,-i\gw_q,0)\in \C^{(2q+1)\times (2q+1)},
}
then $(\mc{G}_1,\mc{G}_2,K)$ can be chosen to be real matrices.  Indeed, in this case we can choose
\eq{
\mc{G}_1 = \diag\left( G_1^1,\ldots,G_1^q,0_{p\times p} \right),
\qquad  G_1^k =  \pmat{0&\gw_k I_Y \\ -\gw_k I_Y& 0} ,
}
\ieq{
K=\eps (K_0^1,\ldots,K_0^q,K_0^{q+1})
}
where $K_0^k = (\re P(i\gw_k)\pinv, \im P(i\gw_k)\pinv)\in \R^{m\times 2p}$ for $k\in \List{q}$ and $K_0^{p+1}=P(0)\pinv\in\R^{m\times p}$, and finally
\ieq{
\mc{G}_2 = \left( G_2^k \right)_{k=1}^{p+1}
}
where $G_2^k = \bigl({-I_Y\atop 0}\bigr)\in \R^{2p\times p}$ for $k\in \List{q}$ and $G_2^{p+1}=-I_Y\in \R^{p\times p}$.
The controller incorporates a p-copy internal model of the exosystem, and 
if we apply a unitary similarity transformation
\eq{
Q=\diag(Q_0,\ldots,Q_0,I_Y), \qquad 
Q_0 = \frac{1}{\sqrt{2}}
\pmat{I_Y&I_Y \\iI_Y&-iI_Y},
}
then $(Q^\ast \mc{G}_1Q,Q^\ast \mc{G}_2,KQ)$ coincides with the controller constructed in the beginning of this section. From this it follows that there exists $\eps^\ast>0$ such that the closed-loop system is exponentially stable and the real controller solves the robust output regulation problem for all $0<\eps\leq \eps^\ast$.  
\end{remark}

\subsection{Controller With a Reduced Order Internal Model}
\label{sec:ROIMcontr}

In this section we construct a minimal order controller for a version of the robust output regulation problem where the controller is only required to tolerate uncertainties from a given class $\Ops_0$ of admissible perturbations~\cite{PauPoh13a,PauPoh14a}. More precisely, in part~(c) of the robust output regulation problem we only consider perturbations such that $\sysops\in \Ops_0$ and for which the perturbed closed-loop system is exponentially stable.
We again assume that the plant is exponentially stable, the matrix $S$ is diagonal, and we in addition assume that $P(i\gw_k)$ are boundedly invertible for all $k\in \List{q}$.

The class $\Ops_0$ 
in the control problem 
can be chosen freely, but it is assumed that all perturbations
$\sysops$
in $\Ops_0$
are such that (i) the perturbed plant $(\tilde{A},\tilde{B},\tilde{C},\tilde{D})$ is a regular linear system and (ii) $i\gw_k\in \rho(\tilde{A})$ and 
the transfer function
$\tilde{P}(i\gw_k) = \tilde{C}_\Lambda R(i\gw_k,\tilde{A})\tilde{B}+\tilde{D}$ 
is boundedly invertible for all $k\in \List{q}$.
Both of these requirements are in particular satisfied for sufficiently small bounded perturbations of $A$, $B$, $C$, and $D$.
Being given such a class $\Ops_0$, we begin the construction of the controller by defining 
\eq{
\mc{S}_k =
\Span 
\bigl\{
\tilde{P}(i\gw_k)\inv (\tilde{C}R(i\gw_k,\tilde{A})\tilde{E}e_k+\tilde{F}e_k)
\,\bigm|\,
\sysops\in \Ops_0
\bigr\}
\subset U
}
for $k\in \List{q}$, where $(e_k)_{k=1}^q$ is the Euclidean basis of $W=\C^q$.
We further define $p_k = \dim \mc{S}_k$. 
The controller that we construct contains a reduced order internal model 
where the number of copies of each of the frequencies $i\gw_k$ of the exosystem   is exactly~$p_k$.
It should be noted that
this controller
differs from the minimial order controller with a full internal model only in the situation where $p_k<\dim Y$ for at least one $k\in \List{q}$.

Define
$Z = Y_1\times \cdots \times Y_q $ where $Y_k = \C^{p_k}$ if $p_k<\dim Y$ and $Y_k=Y$ if $p_k=\dim Y$ or $p_k=\infty$. We choose
\eq{
\mc{G}_1 = \diag \bigl( i\gw_1 I_{Y_1},\ldots,i\gw_q I_{Y_q} \bigr)\in \Lin(Z),
\qquad K = \eps K_0 = \eps\bigl(K_0^1,\ldots,K_0^q\bigr)\in \Lin(Z,U)
}
where $\eps>0$ and $K_0^k\in \Lin(Y_k,U)$ are such that
\eq{
K_0^k = 
\left\{
\begin{array}{ll}
  (u_k^1,\ldots,u_k^{p_k}) &\mbox{if}~~ p_k<\dim Y\\[.7ex]
  P(i\gw_k)\inv &\mbox{if}~~ p_k=\dim Y ~ \mbox{or} ~ p_k=\infty
\end{array}
\right.
}
where $\set{u_k^l}_{l=1}^{p_k}\subset U$ is a basis of the subspace $\mc{S}_k$.
Finally, 
we choose 
\eq{
\mc{G}_2
= (-(P(i\gw_k)K_0^k)^\ast)_{k=1}^q
\in \Lin(Y,Z).
}
For those $k\in \List{q}$ for which $p_k=\dim Y$ or $p_k=\infty$ we then have $\mc{G}_2^k=-I_Y$.

\begin{theorem}
  \label{thm:ContrROIM}
  Assume that the semigroup $T(t)$ generated by $A$ is exponentially stable, $S = \diag(i\gw_1,\ldots,i\gw_q)$,
  and $P(i\gw_k)$ are boundedly invertible for all $k\in \List{q}$.
 Then there exists $\eps^\ast>0$ such that for any $0<\eps\leq \eps^\ast$ the controller with the above choices of parameters solves the robust output regulation problem for the class $\Ops_0$
  of perturbations.
\end{theorem}

\begin{proof}
Let $\eps>0$, $\sysops\in \Ops_0$ and $k\in \List{q}$ 
be arbitrary
and denote $y_k = -\tilde{C} R(i\gw_k, \tilde{A}) \tilde{E}e_k -\tilde{F}e_k$.
We begin by showing that
  there exists $z \in  \ker(i\gw_k-\mc{G}_1)$ such that
    \ieq{
    \tilde{P}(i\gw_k)K z = y_k.
    }
If $k$ is such that $Y_k = Y$, we can choose $z=(z_1,\ldots,z_q)\in Z$ such that $z_l=0$ for $l\neq k$ and 
$z_k =\frac{1}{\eps} P(i\gw_k)\tilde{P}(i\gw_k)\inv y_k$.
Then clearly $z\in \ker(i\gw_k-\mc{G}_1)$ and $\tilde{P}(i\gw_k)Kz = \eps \tilde{P}(i\gw_k)K_0^kz_k= \tilde{P}(i\gw_k)\tilde{P}(i\gw_k)\inv y_k=y_k$.
It remains to consider the situation
$p_k<\dim Y$.
  Since $\tilde{P}(i\gw_k)\inv y_k\in \mc{S}_k$ by definition, and since $\set{u_k^1,\ldots,u_k^{p_k}}$ is a basis of $\mc{S}_k$, there exist $\set{\ga_l}_{l=1}^{p_k}\subset \C$ such that 
  \eq{
  \tilde{P}(i\gw_k)\inv y_k = \sum_{l=1}^{p_k} \ga_l u_k^l.
  }
Choose $z=(z_1,\ldots,z_q)$ such that $z_l=0$ for $l\neq k$ and $z_k= \frac{1}{\eps} (\ga_l)_{l=1}^{p_k}\in Y_k=\C^{p_k}$. Then clearly $z\in \ker(i\gw_k-\mc{G}_1)$ and
  \eq{
  \tilde{P}(i\gw_k) Kz 
  =\eps \tilde{P}(i\gw_k) K_0^k z_k
  =\tilde{P}(i\gw_k) \sum_{l=1}^{p_k}\ga_l u_k^l 
  =\tilde{P}(i\gw_k)\tilde{P}(i\gw_k)\inv y_k
  =y_k.
  }
  Since $k\in \List{q}$ and $\sysops\in \Ops_0$ were arbitrary, we have from~\cite[Thm. 5.1]{PauPoh14a} that the controller solves the robust output regulation problem for the class~$\Ops_0$ of perturbations if the closed-loop system is exponentially stable
  (see the proof of Theorem~\ref{thm:RORPchar}).

It remains to show that there exists $\eps^\ast>0$ such that for every $0<\eps\leq \eps^\ast$ the closed-loop system is exponentially stable.
However, if we define $H=(H_1,\ldots,H_q)\in \Lin(Z,X)$ by choosing $H_k = R(i\gw_k,A_{-1})BK_0^k$, then $H$ is the solution of the Sylvester equation $H \mc{G}_1 = A_{-1}H+BK_0$, and the stability closed-loop system can be established exactly as in the proof of Theorem~\ref{thm:SimpleContrMain}, since we again have $\CL H+DK_0=-\mc{G}_2^\ast$.
\end{proof}

\section{The New Robust Controller Structure}
\label{sec:ContrOne}

In this section we introduce the new controller structure for robust output regulation of linear systems. 
This controller has the natural structure for the inclusion of a p-copy internal model into the dynamics of the controller.
The construction of the controller 
is completed in steps. 
Some of the choices of the parameters require
certain properties from the associated operators, and these properties are verified in Theorem~\ref{thm:ContrOneMain}.
We begin by assuming that $\dim Y<\infty$. The case of an infinite-dimensional output space is considered separately for a diagonal exosystem in Section~\ref{sec:ContrOneDiagExo}.

\medskip

\noindent\textbf{Step $\bm{1}^\circ$:}
We begin by choosing the state space of the controller as $Z=\ZI\times X$, and choosing the general structure of the operators $(\mc{G}_1,\mc{G}_2,K)$ as
  \eq{
  \mc{G}_{1}=\pmat{G_1&G_2(\CL +D\KtL )\\0&\hspace{-1.5ex}A_{-1}+B\KtL+L(\CL +D\KtL)}, 
\qquad  \mc{G}_{2}=\pmat{G_2\\L},
  }
  and $K = (K_1, \; -\KtL)$. The operator $G_1$ is the internal model of the exosystem~\eqref{eq:exointro}, and it is defined by choosing
  $\ZI = Y^{n_1}\times \cdots \times Y^{n_q}$, and
\eq{
G_1 = \diag \bigl( \JY[1],\ldots,\JY[q] \bigr)\in \Lin(\ZI),
\qquad K_1 = \bigl(K_1^1,\ldots,K_1^q\bigr)
\in \Lin(\ZI,U).
}
Here for each $k\in \List{q}$ we have
\eqn{
\label{eq:JYdef}
\JY =
\pmat{i\gw_k I_Y&I_Y\\&i\gw_k I_Y&\ddots\\&&\ddots&I_Y\\ &&&i\gw_k I_Y}
\in \Lin(Y^{n_k}),
}
and $ K_1^k = \bigl(K_1^{k1},\ldots,K_1^{kn_k}\bigr)\in \Lin(Y^{n_k},U)$, where
$n_k\in \N$ is the dimension of the Jordan block in $S$ associated to the eigenvalue $i\gw_k\in \gs(S)$.
 We choose the components $K_1^{k1}\in \Lin(Y,U)$ of each $K_1^k$ in such a way that $P(i\gw_k)K_1^{k1}\in \Lin(Y)$ are boundedly invertible. This is possible since $P(i\gw_k)$ are surjective by assumption, and can be achieved, for example, by choosing $K_1^{k1}=P(i\gw_k)\pinv$ for all $k\in \List{q}$.
  For $l\geq 2$
  we can choose
  $K_1^{kl}$
   freely, e.g., $K_1^{kl}=0$.

\medskip

\noindent\textbf{Step $\bm{2}^\circ$:} 
By Assumption~\ref{ass:stabdetect} we can choose
 $K_2\in\Lin(X_1,U)$ and $L_1\in\Lin(Y,X)$ in such a way that $(A_{-1}+B\KtL)\vert_X$ (here $\KtL$ is the $\Lambda$-extension of $K_2$) and $A+L_1\CL $ generate exponentially stable semigroups. For $\gl\in \rho(A+L_1\CL)$ we define
 \eq{
 P_L(\gl) = \CL R(\gl,A_{-1}+L_1\CL)(B+L_1 D) + D.
 }
The identity $P_L(i\gw_k) = (I-\CL R(i\gw_k,A)L_1)\inv P(i\gw_k)$ and the choice of $K_1$ imply that  $P_L(i\gw_k)K_1^{1k}\in \Lin(Y)$ are boundedly invertible for all $k\in \List{q}$.
    The domain of the operator $\mc{G}_1$ is chosen as
    \eq{
    \Dom(\mc{G}_1) 
    &= \Setm{(z_1,x_1)\in \ZI\times X_B}{ A_{-1}x_1+B\KtL x_1\in X}.
    }

\medskip

\noindent\textbf{Step $\bm{3}^\circ$:} 
We define  $H = (H_1,H_2,\ldots,H_q)\in \Lin(\ZI,X)$ where $H_k=(H_k^1,H_k^2,\ldots,H_k^{n_k})\in \Lin(Y^{n_k},X)$ and
\eq{
H_k^l = \sum_{j=1}^l (-1)^{l-j} R(i\gw_k,A_{-1}+L_1\CL)^{l+1-j} (B+L_1D) K_1^{kj}.
}
We have from~\cite[Sec. 7]{Wei94} that $(A+L_1\CL,B+L_1D,\CL,D)$ is a regular linear system, and the resolvent identity in~\cite[Prop. 6.6]{Wei94}
implies $\ran(H)\subset X_B$. We can therefore define $C_1=\CL H+DK_1\in \Lin(\ZI,Y)$.

\medskip

\noindent\textbf{Step $\bm{4}^\circ$:}
We choose  $G_2\in \Lin(Y,\ZI)$ in such a way that the semigroup generated by $G_1+G_2C_1\in \Lin(\ZI)$ is exponentially stable (i.e., the matrix is Hurwitz). The detectability of the pair $(C_1,G_1)$ is proved in Theorem~\ref{thm:ContrOneMain} below. Finally, we define $L=L_1 + HG_2\in \Lin(Y,Z)$.

\begin{theorem}
  \label{thm:ContrOneMain}
Assume $\dim Y<\infty$. The controller with the above choices of parameters solves the robust output regulation problem.

  In particular, the controller $(\mc{G}_1,\mc{G}_2,K)$ has the following properties:
  \begin{itemize} 
     \item[\textup{(i)}] The operator $\mc{G}_1$ generates a semigroup on $Z$ and the controller $(\mc{G}_1,\mc{G}_2,K)$ 
       incorporates a p-copy internal model of the exosystem.
\item[\textup{(ii)}] The operator $H$ is the unique solution of the Sylvester equation
  \eqn{
  \label{eq:RORPCLsysstabSyl}
  HG_1 =(A_{-1}+L_1\CL )H+(B+L_1D)K_1,
  } 
  and the pair $(C_1,G_1)$ where $C_1 = \CL H + DK_1\in \Lin(\ZI,Y)$ is exponentially detectable.  
    \item[\textup{(iii)}] The semigroup generated by 
      $A_e$ is exponentially stable.
  \end{itemize}
\end{theorem}

\begin{proof} 
  We begin by proving part (i).
  We have that
  \eq{
  \mc{G}_{1}
  &=\pmat{G_1&0\\0&A_{-1}}
  +\pmat{0&G_2\\B&L}\pmat{I&0\\D&I}\pmat{0&\KtL \\0&\CL}
  }
  where $\pmatsmall{0&G_2\\B&L}$ and $\pmatsmall{0&\KtL\\0&\CL}$ 
  are admissible with respect to $\pmatsmall{G_1&0\\0&A}$. 
  It is now straightforward to use the results in~\cite[Sec. 7]{Wei94} to verify that $\mc{G}_1$
with the proposed domain
  generates a strongly continuous semigroup on $Z=\ZI\times X$.
  Moreover, it is easy to show that $\KL=K = (K_1,\;-\KtL)$.
For every $k\in \List{q}$ the matrix $G_1$ clearly satisfies $\dim\ker(i\gw_k-G_1)=\dim Y=p$
and it has exactly $p$ 
Jordan blocks
of size $n_k\times n_k$
associated to $i\gw_k$. 
Due to the triangular structure of $\mc{G}_1$, the controller 
therefore
incorporates a p-copy internal model of the exosystem.

We will now show that $H$ is the solution of the Sylvester equation~\eqref{eq:RORPCLsysstabSyl}.
Denote $A_L=A_{-1}+L_1\CL$ and $B_L=B+L_1D$ for brevity.
Due to the structure of the operator $G_1$ it is straigtforward to see that an operator $H\in \Lin(\ZI,X)$ such that $\ran(H)\subset \Dom(\CL)$ is the solution of $HG_1 = A_LH+B_LK_1$ if and only if for all $k\in \List{q}$ we have
\eq{
(i\gw_k-A_L)H_k^1 &= B_LK_1^{k1}\\
(i\gw_k-A_L)H_k^2 + H_k^1 &= B_LK_1^{k2}\\
&~\;\vdots\\
(i\gw_k-A_L)H_k^{n_k} + H_k^{n_k-1} &= B_LK_1^{kn_k},
}
where $H=(H_1,\ldots,H_q)$, and $H_k=(H_k^1,\ldots,H_k^{n_k})\in \Lin(Y^{n_k},X)$.
For each $k\in \List{q}$ the above system of equations has a unique solution
\eq{
H_k^l = \sum_{j=1}^l (-1)^{l-j} R(i\gw_k,A_L)^{l+1-j} B_L K_1^{kj}.  
}
Thus $H$ defined in Step~3$^\circ$ is the unique solution of~\eqref{eq:RORPCLsysstabSyl}.

We will now show that $(C_1,G_1)$ is exponentially detectable.  We can do this by showing that for all $k\in\List{q}$ and $z\in \ker(i\gw_k-G_1)$ with $z\neq 0$ we have $C_1 z\neq 0$~\cite[Thm. 6.2-5]{Kai80book}. To this end, let $k\in\List{q}$ and $z\in \ker(i\gw_k-G_1)$ such that $z\neq 0$ be arbitrary. From the structure of $G_1$ we have that
$z=(z_1,\ldots,z_q)$ where $z_l=0$ for $l\neq k$, and further $z_k=(z_k^1,0,\ldots,0)\in Y^{n_k}$.
Using $H_k^1=R(i\gw_k,A_L)B_LK_1^{k1}$ 
we see that
\eq{
C_1 z 
= \CL Hz + DK_1z
= \CL H_kz_k + DK_1^kz_k
= \CL H_k^1z_k^1 + DK_1^{k1}z_k^1
= P_L(i\gw_k)K_1^{k1}z_k^1 \neq 0
}
since $z_k^1\neq 0$, and since we chose $K_1^{k1}$ in such a way that $P(i\gw_k)K_1^{k1}$ and $P_L(i\gw_k)K_1^{k1}$ are boundedly invertible.

It remains to 
show that the closed-loop system is exponentially stable. 
With the chosen controller $(\mc{G}_1,\mc{G}_2,K)$ 
the operator $A_e$ becomes
\eq{
A_{e}
=\pmat{A_{-1}&BK_{1}&-B\KtL\\G_2\CL &G_1+G_2DK_1&G_2\CL \\L\CL &LDK_{1}&A_{-1}+B\KtL+L\CL }
}
with domain $\Dom(A_e)$ equal to
\eq{
\Dom(A_e) = 
\biggl\{ 
(x,z_1,x_1)\in  X_B\times \ZI\times X_B
~\biggm|
~
\Bigl\{
\begin{array}{l}
  A_{-1}x+BK_1z_1-B\KtL x_1\in X\\
  A_{-1}x_1+B\KtL x_1\in X
\end{array}
\biggr\}.
} 
If we choose a similarity transform $Q_{e}\in \Lin(X\times \ZI\times X)$ 
\eq{
Q_{e}
=\pmat{I&0&0\\0&I&0\\-I&H&-I}
=  Q_{e}\inv ,
}
we can define $\hat{A}_e = Q_{e}A_{e}Q_{e}\inv$ on $X\times \ZI\times X$. If we denote $x_e = (x,z_1,x_1)\in X\times \ZI \times X$ and use $\ran(H)\subset X_B$, the domain of the operator $\hat{A}_e$ satisfies
\eq{
\Dom(\hat{A}_e) 
&= \Setm{ x_e \in X\times \ZI\times X}{Q_e\inv x_e \in \Dom(A_e)}\\
&= \Setm{ x_e \in 
X_B\times \ZI\times X_B
}{Q_e\inv x_e \in \Dom(A_e)}.
}
For $x_e = (x,z_1,x_1) \in 
X_B\times \ZI\times X_B$
we thus have
\eq{
Q_e\inv x_e \in \Dom(A_e) \quad
 \Leftrightarrow  \quad & 
\left\{
\begin{array}{l}
  A_{-1}x+BK_1z_1-B\KtL(-x+Hz_1-x_1)\in X\\
  (A_{-1}+B\KtL)(-x+Hz_1-x_1)\in X
\end{array}
\right.\\
\Leftrightarrow \quad & 
\left\{
\begin{array}{l}
  (A_{-1}+B\KtL)x+B(K_1-\KtL H)z_1+B\KtL x_1\in X\\
   BK_1z_1 + A_{-1}Hz_1-A_{-1}x_1 \in X
\end{array}
\right.
\\
\Leftrightarrow \quad & 
\left\{
\begin{array}{l}
  (A_{-1}+B\KtL )x+B(K_1-\KtL H)z_1+B\KtL x_1\in X\\
   x_1 \in \Dom(A)
\end{array}
\right.
}
since equation~\eqref{eq:RORPCLsysstabSyl} implies
$A_{-1}Hz_1+BK_1z =HG_1z_1-L_1(\CL H+DK_1)z_1 \in X$. The above conditions also imply $x\in X_B$, and thus 
\eq{
\Dom(\hat{A}_e) 
= \bigl\{ x_e \in  X\times \ZI\times \Dom(A)
\,\bigm|\,(A_{-1}+B\KtL )x+
B(K_1-\KtL H)z_1+B\KtL x_1\in X
\bigr\}.
} 
For $x_e=(x,z_1,x_1)\in  \Dom(\hat{A}_e)$ 
a direct computation using $L=L_1+G_2H$, $C_1=\CL H+DK_1$, and $HG_1z_1 =(A_{-1}+L_1\CL )Hz_1+(B+L_1D)K_1z_1$ yields
\eq{
\hat{A}_{e}x_e
&=Q_{e}A_{e} \pmat{x\\z_1\\-x+Hz_1-x_1}
=\pmat{(A_{-1} + B\KtL )x +B(K_{1}-\KtL H)z_1 +B\KtL  x_1\\(G_1+G_2(\CL H+DK_1))z_1 - G_2\CL  x_1\\ (A_{-1} +L_1\CL)  x_1}\\
&=\pmat{A_{-1} + B\KtL  & B(K_{1}-\KtL H) &B\KtL  \\0& G_1+G_2C_1 &- G_2\CL  \\ 0&0& A +L_1\CL} \pmat{x\\z_1\\   x_1}
}
The operator $G_2 \in \Lin(Y,\ZI)$ was chosen in such a way that $G_1+G_2C_1\in \Lin(\ZI)$ is Hurwitz. 
Since $(A_{-1}+B\KtL )\vert_X$ and $A+L_1\CL $  generate exponentially stable semigroups,
since $B$ is an admissible input operator for $(A+B\KtL )\vert_X$, $\CL$ and $\KtL$ are  admissible input operators for $A+L_1\CL$, and $K_1-\KtL H$ and $G_2$ are bounded,
we have that the semigroup generated by $\hat{A}_e$ is exponentially stable, and due to similarity, the same is true for $A_e$. We thus conclude that the closed-loop system is exponentially stable.

Because the controller incorporates a p-copy internal model of the exosystem and the closed-loop system is exponentially stable, we have from Theorem~\ref{thm:RORPchar} that the controller solves the robust output regulation problem.
\end{proof}

\subsection{Controller for a Diagonal Exosystem}
\label{sec:ContrOneDiagExo}

In this section we consider the situation where the output space $Y$ is allowed to be infinite-dimensional and the matrix $S$ in the exosystem is diagonal. We will show that
in this situation
the robust output regulation problem can be solved with particularly simple choice for the parameter $G_2$ of the controller. For a diagonal matrix $S=\diag(i\gw_1,\ldots,i\gw_q)$ we choose $\ZI = Y^q$ and the internal model $(G_1,K_1)$ of the exosystem is defined as
\eq{
G_1 = \diag(i\gw_1 I_Y,\ldots,i\gw_q I_Y),
\qquad K_1 = (K_1^1,\ldots,K_1^q)
}
where  $K_1^k$ are chosen in such a way that $P(i\gw_k)K_1^k$ are boundedly invertible for all $k\in \List{q}$.
The following is the main result of this section.

\begin{theorem}
  \label{thm:ContrOneDiagExo}
  Assume $S=\diag(i\gw_1,\ldots,i\gw_q)$. If the other parameters of the controller are chosen as in the beginning of Section~\textup{\ref{sec:ContrOne}} and if we choose
  \eq{
G_2 
= (G_2^k)_{k=1}^q
=   (-(P_L(i\gw_k)K_1^k)^\ast)_{k=1}^q \in \Lin(Y,\ZI),
  }
then the controller solves the robust output regulation problem.  

If we choose $K_1^k = P_L(i\gw_k)\pinv =  P(i\gw_k)\pinv (I-\CL R(i\gw_k,A)L_1)$ for all $k$, then $G_2^k=-I_Y$ for all $k$.
\end{theorem}

\begin{proof}
  Since $n_k=1$ for all $k\in \List{q}$, we have $H=(H_1,\ldots,H_q)\in \Lin(\ZI,X)$, where $H_k=R(i\gw_k,A_{-1}+L_1\CL)(B+L_1D)K_1^k$. Because of this, the operator $C_1 = (C_1^1,\ldots,C_1^q)$ satisfies
  \eq{
  C_1^k  
  = \CL H_k+DK_1^k 
  = P_L(i\gw_k)K_1^k
  }
  for all $k\in \List{q}$, which shows that
   $G_2 = -C_1^\ast$. The last claim of the theorem follows immediately from
  $P_L(i\gw_k)=(I-\CL R(i\gw_k,A)L_1)\inv P(i\gw_k)$. The same identity and the fact that $K_1^k$ were chosen so that $P(i\gw_k)K_1^k$ are boundedly invertible imply that the components $G_2^k$ of $G_2$ are boundedly invertible for all $k\in \List{q}$.
  We thus have from Lemma~\ref{lem:G1fbstab} that the semigroup generated by $G_1+G_2C_1=G_1-G_2 G_2^\ast$ is exponentially stable.
  The exponential stability of the closed-loop system can now be shown exactly as in the proof of Theorem~\ref{thm:ContrTwoMain}.

  Due to the fact that $Y$ may be infinite-dimensional, we cannot use the concept of p-copy internal model. Instead, we will verify that the controller satisfies the \Gconds. For this we will in particular use Lemma~\ref{lem:Gcondsinvariance}.

  Since $S$ is diagonal, the condition~\eqref{eq:Gconds3} is trivially satisfied.
  The components $G_2^k = -(P_L(i\gw_k)K_1^k)^\ast$ of 
  $G_2=(G_2^k)_{k=1}^q$ are boundedly invertible for all $k\in \List{q}$.
This implies $\ker(G_2)=\set{0}$, and also further shows that $\ker(\mc{G}_2)=\set{0}$.
Moreover, if for some $k\in\List{q}$ the elements $(z,x)\in Z$, $(w,v)\in \Dom(\mc{G}_1)$ with $w=(w_k)_{k=1}^q\in \ZI$, and $y\in Y$ are such that
\eq{
\pmat{z\\x}
=\pmat{i\gw_k-G_1&0\\0&\hspace{-.4cm} i\gw_k-(A_{-1}+B\KtL)} \pmat{w\\v} = \pmat{G_2\\L}y,
}
then we in particular have
$z=(i\gw_k-G_1)w=G_2y$ and $G_2^ky=(i\gw_k-i\gw_k)w_k=0$. The invertibility of $G_2^k$ implies $y=0$ and $(z,x)=\mc{G}_2 y=0$. Since $k\in\List{q}$ was arbitrary, this shows that the operators $(\pmatsmall{G_1&0\\0&A_{-1}+B\KtL},\mc{G}_2)$ satisfy the \Gconds.
Since
  \eq{
  \mc{G}_{1}
  =
\pmat{G_1&G_2(\CL +D\KtL)\\0&\hspace{-.7ex}A_{-1}+B\KtL+L(\CL +D\KtL)}
  =\pmat{G_1&0\\0&\hspace{-.9ex}A_{-1}+B\KtL}
  +\pmat{G_2\\L}\pmat{0&\hspace{-.5ex}\CL +D\KtL}
  }
  where for any $k\in\List{q}$ we have $\ker( i\gw_k-\pmatsmall{G_1&0\\0&A_{-1}+B\KtL})\subset  \ZI\times \set{0}\subset \ker( (0,\CL+D\KtL))$, Lemma~\ref{lem:Gcondsinvariance} shows that the operators $(\mc{G}_1,\mc{G}_2)$ satisfy the
  \Gconds\ 
  as well. 

  Since the controller satisfies the \Gconds\ and the closed-loop system is exponentially stable, we have from Theorem~\ref{thm:RORPchar} that the controller solves the robust output regulation problem.
\end{proof}

\subsection{Controller with a Reduced Order Internal Model}

It was shown in~\cite{Pau15a} that the triangular structure used in this section is ideal for controllers with reduced order internal models.  Indeed, if the internal model $(G_1,K_1)$ is replaced with an appropriate reduced order internal model, the controller will solve the robust output regulation problem for a given class $\Ops_0$ of perturbations.  As the final result in this section we present a generalization of the controller introduced in~\cite{Pau15a} for regular linear systems with diagonal exosystems.
For this purpose we again assume that $P(i\gw_k)$ are invertible for all $k\in\List{q}$.

Let $\Ops_0$ be a class of admissible perturbations.
Similarly as in Section~\ref{sec:ROIMcontr}
we define
$\ZI = Y_1\times \cdots \times Y_q $, and 
\eq{
G_1 = \diag \bigl( i\gw_1 I_{Y_1},\ldots,i\gw_q I_{Y_q} \bigr)
\in \Lin(\ZI),
\qquad K_1 =\bigl(K_1^1,\ldots,K_1^q\bigr)
\in \Lin(\ZI,U)
}
where $K_1^k\in \Lin(Y_k,U)$ are such that
\eq{
K_1^k = 
\left\{
\begin{array}{ll}
  (u_k^1,\ldots,u_k^{p_k}) &\mbox{if}~~ p_k<\dim Y\\[.7ex]
  P(i\gw_k)\inv &\mbox{if}~~ p_k=\dim Y ~ \mbox{or} ~ p_k=\infty
\end{array}
\right.
}
in the notation of Section~\textup{\ref{sec:ROIMcontr}}.
 Moreover, we define
  \ieq{
  G_2 = (-(P_L(i\gw_k)K_1^k)^\ast)_{k=1}^q \in \Lin(Y,Z).
  }
  The rest of the parameters of the controller $(\mc{G}_1,\mc{G}_2,K)$ are chosen as in the beginning of Section~\textup{\ref{sec:ContrTwo}}.

\begin{theorem}
  \label{thm:ContrOneROIM}
  Assume
  $S = \diag(i\gw_1,\ldots,i\gw_q)$ and $P(i\gw_k)$ are invertible for all $k\in \List{q}$. 
  Then
the controller with the above choices of parameters solves the robust output regulation problem for the class $\Ops_0$ of perturbations.
\end{theorem}

\begin{proof}
If $\sysops\in \Ops_0$ and $k\in \List{q}$, and if
we choose $z$ as in the proof of Theorem~\ref{thm:ContrROIM} (for $\eps=1$), then it is easy to see that $\tilde{P}(i\gw_k)K \left( z\atop 0 \right)=y_k$ and $\left( z\atop 0 \right)\in \ker(i\gw_k-\mc{G}_1)$. By~\cite[Thm. 5.1]{PauPoh14a} the controller solves the robust output regulation problem for the class $\Ops_0$ of perturbations provided that the closed-loop system is exponentially stable. The stability of the closed-loop system can be shown exactly as in the proof of Theorem~\ref{thm:ContrOneDiagExo}.
\end{proof}

\section{The Observer-Based Robust Controller}
\label{sec:ContrTwo}

The observer-based robust controller structure presented in this section is based on the controller H\"{a}m\"{a}l\"{a}inen and Pohjolainen~\cite{HamPoh10} for systems with bounded input and output operators. The construction of the controller is again completed in steps and its properties are given in Theorem~\ref{thm:ContrTwoMain}. For this controller structure it is necessary to assume that the spaces $U$ and $Y$ are isomorphic. We begin by assuming that the plant has the same finite number of inputs and outputs, that is, $U=Y=\C^p$. The case of an infinite-dimensional output space is again considered separately for a diagonal exosystem in Theorem~\ref{thm:ContrTwoDiagExo}.

\medskip

\noindent\textbf{Step $\bm{1}^\circ$:}
We begin by choosing the state space of the controller as $Z=\ZI\times X$, and choosing 
  \eq{
  \mc{G}_{1}=\pmat{G_1&0\\(B+LD)K_1&A_{-1}+BK_2+L(\CL +DK_2)}, 
  }
$\mc{G}_{2}=\left(G_2\atop-L\right)$,
  and $K = (K_1, \; \KtotL)$. The operators $(G_1,G_2)$ make up the internal model of the exosystem~\eqref{eq:exointro}, and they are defined by choosing
  $\ZI = Y^{n_1}\times \cdots \times Y^{n_q}$, and
\eq{
G_1 = \diag \bigl( \JY[1],\ldots,\JY[q] \bigr)\in \Lin(\ZI),
\qquad 
G_2 =
(G_2^k)_{k=1}^q
\in \Lin(Y,\ZI).
}
Here 
$\JY$ are as in~\eqref{eq:JYdef} and
$G_2^k 
= (G_2^{kl})_{l=1}^{n_k}
\in \Lin(Y,Y^{n_k})$ for all $k\in \List{q}$,
where $n_k\in \N$ is the dimension of the Jordan block in $S$ associated to the eigenvalue $i\gw_k\in \gs(S)$. We choose the components $G_2^{kn_k}\in \Lin(Y)$ of each $G_2^k$ to be boundedly invertible (e.g., it is possible to choose $G_2^{kn_k}=I_Y$ for every $k\in \List{q}$).  

\medskip

\noindent\textbf{Step $\bm{2}^\circ$:} 
By Assumption~\ref{ass:stabdetect} 
we can choose $K_{21}\in\Lin(X_1,U)$ and $L\in\Lin(Y,X)$  in such a way that $(A_{-1}+B\KtoL)\vert_X$ (here $\KtoL$ is the $\Lambda$-extension of $K_{21}$) and $A+L\CL $ generate exponentially stable semigroups. 
For $\gl\in \rho( A_{-1}+B\KtoL )$ we define
 \eq{
P_K(\gl) = (\CL + D\KtoL ) R(\gl,A_{-1}+B\KtoL )B + D.
 }
Since $P(i\gw_k)$ were assumed to be surjective for all $k\in \List{q}$ and since $U=Y=\C^q$, the identity $P_K(i\gw_k) = P(i\gw_k)(I-\KtoL R(i\gw_k,A_{-1})B)\inv$
implies that $P_K(i\gw_k)$ are boundedly invertible for all $k\in \List{q}$.

\medskip

\noindent\textbf{Step $\bm{3}^\circ$:} 
We define an operator $H: \Dom(H)\subset X_{-1}\to \ZI$
in such a way that
\ieq{
H = (H_k)_{k=1}^q
}
and
\ieq{
H_k = (H_k^l)_{l=1}^{n_k},
}
where
\eq{
H_k^l = \sum_{j=l}^{n_k} (-1)^{j-l} G_2^{kj}(\CL\hspace{-.3ex}+\hspace{-.5ex}D\KtoL )R(i\gw_k,A_{-1}\hspace{-.3ex}+ \hspace{-.3ex}B\KtoL )^{j+1-l} .
}
Since we have from~\cite[Sec. 7]{Wei94} that $(A+B\KtoL ,B,\CL+D\KtoL ,D)$ is a regular linear system
and $X_B\subset \Dom(\CL)\cap \Dom(\KtoL)$,
it is immediate that $H\in \Lin(X,\ZI)$ and $\ran(B)\subset \Dom(H)$, and we can thus
define $B_1= HB+G_2D\in \Lin(U,\ZI)$.

\medskip

\noindent\textbf{Step $\bm{4}^\circ$:}
We choose the operator $K_1\in \Lin(\ZI,U)$ in such a way that the semigroup generated by $G_1+B_1K_1\in \Lin(\ZI)$ is exponentially stable (i.e., the matrix is Hurwitz). The stabilizability of the pair $(G_1,B_1)$ is shown in Theorem~\ref{thm:ContrTwoMain} below. Finally, we define $\KtotL =\KtoL  + K_1H\in \Lin(X,U)$ and choose the domain of the operator $\mc{G}_1$ as
\eq{
\Dom(\mc{G}_1) 
    =
    \bigl\{
(z_1,x_1)\in \ZI\times X_B
\,\bigm|\,
    A_{-1}x_1+B(K_1z_1+\KtotL x_1)\in X
    \bigr\}.
}

\begin{theorem}
  \label{thm:ContrTwoMain}
Assume $U = Y = \C^p$.
  The controller with the above choices of parameters solves the robust output regulation problem.

  In particular, the controller $(\mc{G}_1,\mc{G}_2,K)$ has the following properties:
  \begin{itemize} 
     \item[\textup{(i)}] The operator $\mc{G}_1$ generates a semigroup on $Z$ and the controller $(\mc{G}_1,\mc{G}_2,K)$ satisfies the \Gconds\ in Definition~\textup{\ref{def:Gconds}}.   
\item[\textup{(ii)}] The operator $H$ is the unique solution of the Sylvester equation
  \eqn{
  \label{eq:RORPCLsysstabSylContrTwo}
  G_1H =H(A_{-1}+B\KtoL )+G_2(\CL + D\KtoL ) 
  } 
  on $\Sylspace$.
  Moreover, $(G_1,B_1)$ where $B_1 =  HB + G_2D\in \Lin(U,\ZI)$ is exponentially stabilizable.  
    \item[\textup{(iii)}] The semigroup generated by 
      $A_e$ is exponentially stable.
  \end{itemize}
\end{theorem}

\begin{proof}
  The property that $\mc{G}_1$ with the given domain generates a strongly continuous semigroup can be seen analogously as in the proof of Theorem~\ref{thm:ContrOneMain}.

We will now show that $H$ defined in Step~3$^\circ$ is the solution of~\eqref{eq:RORPCLsysstabSylContrTwo}.
  Denote $A_K=A_{-1}+B\KtoL $ and $C_K=\CL+D\KtoL$ for brevity.
  The structure of  $G_1$ implies that an operator 
$H$
  is the solution of $G_1H = HA_K+G_2C_K$ 
  if and only if 
 $H = (H_k)_{k=1}^q$
 with $H_k 
 = (H_k^l)_{l=1}^{n_k}$ for all $k$, 
  and 
  for all $k\in \List{q}$ we have
\eq{
H_k^1(i\gw_k-A_K) + H_k^2&= G_2^{k1}C_K\\
&~\;\vdots\\
H_k^{n_k-1}(i\gw_k-A_K) + H_k^{n_k}  &= G_2^{k,n_k-1}C_K\\
H_k^{n_k}(i\gw_k-A_K)  &= G_2^{kn_k}C_K
}
on $\Sylspace$.
For every $k\in \List{q}$ the above system of equations has a unique solution which is exactly $H_k$ in step~3$^\circ$.

We will now show that the pair $(G_1,B_1)$ with $B_1 = HB+G_2D$ is exponentially stabilizable. This is equivalent to the pair $(B_1^\ast,G_1^\ast)$ being exponentially detectable. Let $k\in \List{q}$ and $z=(z_1,\ldots,z_q)\in\ker(-i\gw_k-G_1^\ast)$. Then $z_l=0$ for $l\neq k$, and $z_k=(0,\ldots,0,z_k^{n_k})$ with $z_k^{n_k}\in Y$. For any $u\in U$ we have
\eq{
\iprod{u}{B_1^\ast z}
&=\iprod{B_1u}{z}
=\iprod{(H_k^{n_k}B+G_2^{kn_k}D)u}{z_k^{n_k}}
=\iprod{G_2^{kn_k}(C_KR(i\gw_k,A_K)B+D)u}{z_k^{n_k}}\\
&=\iprod{u}{(G_2^{kn_k}P_K(i\gw_k))^\ast z_k^{n_k}}
}
which immediately implies that we can have $B_1^\ast z = 0$ only if $z_k^{n_k}=0$ due to the fact that $G_2^{kn_k}$ and $P_K(i\gw_k)$ are invertible. Since this also implies $z=0$ and since $k\in \List{q}$ was arbitrary, we have that the pair $(G_1,B_1)$ is exponentially stabilizable~\cite[Thm. 6.2-5]{Kai80book}.
Because of this it is possible to choose $K_1$ in such a way that $G_1+B_1K_1$ is Hurwitz.

We will now show that the closed-loop system is exponentially stable. 
When the controller $(\mc{G}_1,\mc{G}_2,K)$ is chosen as suggested, we have that
\eq{
A_{e}
=\pmat{A_{-1}&BK_{1}&B\KtotL\\G_2\CL &G_1+G_2DK_1&G_2D\KtotL  \\-L\CL &BK_{1}&A_{-1}+B\KtotL +L\CL }
}
with domain
\eq{
\Dom(A_e) = \biggl\{
(x,z_1,x_1)\in X_B\times \ZI\times X_B \,\biggm|\,
\Bigl\{
\begin{array}{l}
  A_{-1}x+BK_1z_1+B\KtotL x_1\in X\\
BK_1z_1+ (A_{-1}+B\KtotL )x_1\in X
\end{array}
\biggr\}.
} 
If we choose a similarity transform $Q_{e}\in \Lin(X\times \ZI\times X)$ 
\eq{
Q_{e}
=\mbox{\small$\displaystyle\pmat{-I&0&0\\H&I&0\\-I&0&I}$} 
=  Q_{e}\inv ,
}
we can define $\hat{A}_e = Q_{e}A_{e}Q_{e}\inv$ on $X\times \ZI\times X$. 
If we denote $x_e = (x,z_1,x_1)\in X\times \ZI \times X$, we have
\eq{
\Dom(\hat{A}_e) 
&= \Setm{ x_e \in X\times \ZI\times X}{Q_e\inv x_e \in \Dom(A_e)}\\
&= \Setm{ x_e \in X_B\times \ZI\times X_B}{Q_e\inv x_e \in \Dom(A_e)}
}
and for $x_e =(x,z_1,x_1)\in X_B\times \ZI\times X_B$
we have 
\eq{
Q_e\inv x_e \in \Dom(A_e)
\quad \Leftrightarrow  \quad& 
\left\{
\begin{array}{l}
  -A_{-1}x+BK_1(Hx + z_1)+B\KtotL (-x+x_1)\in X\\
  BK_1(Hx+z_1)+ (A_{-1}+B\KtotL )(-x+x_1)\in X
\end{array}
\right.\\
\Leftrightarrow \quad& 
\left\{
\begin{array}{l}
  (A_{-1}+B\KtoL)x -BK_1z_1 - B\KtotL  x_1\in X\\
x_1\in \Dom(A)
\end{array}
\right.
}
where we have used $\KtoL=\KtotL -K_1H$.
Since the above condition also implies $x\in X_B$, the 
domain of $\hat{A}_e$ becomes
\eq{
\Dom(\hat{A}_e) 
=\bigl\{ x_e \in X\times \ZI\times \Dom(A) \,\bigm|\,
(A_{-1}+B\KtoL)x -BK_1z_1 - B\KtotL  x_1\in X \bigr\}.
} 
For any $x_e=(x,z_1,x_1)\in  \Dom(\hat{A}_e)$ 
a direct computation using $\KtotL  = \KtoL+K_1H$, $B_1=HB+G_2D$, and $G_1Hx =H(A_{-1}+B\KtoL)x+G_2(\CL + D\KtoL)x$ yields
\eq{
\hat{A}_{e}x_e 
&=Q_{e}A_{e}\pmat{-x\\H x+z_1\\-x+x_1}
= \pmat{(A_{-1} + B \KtoL)x -BK_1 z_1 -B\KtotL  x_1\\ (G_1  +(H B+G_2DK_1))z_1 +(HB+G_2D)\KtotL x_1 \\ (A_{-1} + L\CL) x_1}\\
&= \pmat{A_{-1} + B \KtoL & -BK_1 & -B\KtotL\\ 0& G_1  +B_1K_1 &B_1\KtotL\\ 0&0& A_{-1} + L\CL } \pmat{x\\z_1\\x_1}
}
The operator $K_1 \in \Lin(\ZI,U)$ was chosen in such a way that $G_1+B_1K_1\in \Lin(\ZI)$ is Hurwitz. Since $(A_{-1}+B\KtoL)\vert_X$ and $A+L\CL $  generate exponentially stable semigroups,
since $B$ and $\KtotL$ are admissible with respect to $(A+B\KtotL )\vert_X$ and $(A_{-1}+L\CL)\vert_X$, respectively, 
and since $K_1$ and $B_1$ are bounded,
 the semigroup generated by $\hat{A}_e$ is exponentially stable, and because of similarity, the same is also true for $A_e$. This concludes that the closed-loop system is exponentially stable.

It remains to show that the controller satisfies the \Gconds.
We begin by showing that $(G_1,G_2)$ satisfy the \Gconds. We have $\ker(G_2)=\set{0}$ since $G_2^{kn_k}$ are boundedly invertible for all $k\in \List{q}$. 
If $z\in \ran(i\gw_k-G_1)\cap \ran(G_2)$ for some $k\in \List{q}$, there exist $z_1 $ and $y$ such that $z=(i\gw_k-G_1)z_1=G_2y$. Here $z=(z_1,\ldots,z_q)$ with $z_k=(z_k^1,\ldots,z_k^{n_k})\in Y^{n_k}$, and structure of $G_1$ implies that necessarily $z_k^{n_k}=0$. On the other hand, we have $0=z_k^{n_k}=G_2^{kn_k}y$, which implies $y=0$ since $G_2^{kn_k}$ is invertible, and thus $z=G_2y=0$. This concludes that $\ran(i\gw_k-G_1)\cap \ran(G_2)=\set{0}$.
Finally, a direct computation can be used to verify that $\ker(i\gw_k-G_1)^{n_k-1}=\setm{z=(z_1,\ldots,z_q)}{z_k^{n_k}=0, ~ z_l=0 ~ \mbox{for} ~ l\neq k}\subset \ran(i\gw_k-G_1)$. This concludes that $(G_1,G_2)$ satisfy the \Gconds. Moreover, the surjectivity of the operators $G_2^{kn_k}$ implies $\ZI=\ran(i\gw_k-G_1)+\ran(G_2)$, and we thus have $\ZI = \ran(i\gw_k-G_1)\oplus \ran(G_2)$.

We will now show that $(\mc{G}_1,\mc{G}_2)$ satisfy the \Gconds. 
The condition $\ker(\mc{G}_2)=\set{0}$ follows immediately from $\ker(G_2)=\set{0}$.
If $(z,x)\in \ran(i\gw_k-\mc{G}_1)\cap \ran(\mc{G}_2)$, there exist $(z_1,x_1)\in \Dom(\mc{G}_1)$ and $y\in Y$ such that 
\eq{
\pmat{z\\x} \hspace{-.3ex} = \hspace{-.3ex} \pmat{i\gw_k-G_1&0\\B_LK_1& \hspace{-.3cm}A_{-1}+B_L\KtotL+L\CL }\hspace{-.5ex} \pmat{z_1\\x_1} = \hspace{-.3ex}\pmat{G_2\\-L}\hspace{-.3ex}y.
}
where 
we have denoted $B_L = B+LD$.
The first line implies $z\in \ker(i\gw_k-G_1)\cap \ran(G_2)=\set{0}$, and  since $\ker(G_2)=\set{0}$, we have $y=0$. Thus $(z,x)=\mc{G}_2y=0$ and we conclude that $ \ran(i\gw_k-\mc{G}_1)\cap \ran(\mc{G}_2) = \set{0}$.

Finally, let $(z,x)\in \ker(i\gw_k- \mc{G}_1)^{n_k-1}$. Since the closed-loop is exponentially stable, we have $Z=\ran(i\gw_k-\mc{G}_1) + \ran(\mc{G}_2)$ for all $k\in\List{q}$~\cite[Lem. 5.7]{PauPoh10}. 
Thus there exist $(z_1,x_1)\in \Dom(\mc{G}_1)$ and $y\in Y$ such that
\eq{
\pmat{z\\x} \hspace{-.3ex}=\hspace{-.3ex} \pmat{i\gw_k-G_1&0\\B_LK_1& \hspace{-.3cm}A_{-1}+B_L\KtotL +L\CL } \pmat{z_1\\x_1} \hspace{-.3ex}+\hspace{-.3ex} \pmat{G_2\\-L}y.
}
We will show that $y=0$, which will conclude that $(z,x)\in \ran(i\gw_k-\mc{G}_1)$. 
From the above equation we
see that $z=(i\gw_k-G_1)z_1+G_2y$. The property $(z,x)\in \ker(i\gw_k-\mc{G}_1)^{n_k-1}$ and the triangular structure of $\mc{G}_1$ imply $z\in \ker(i\gw_k-G_1)^{n_k-1}\subset \ran(i\gw_k-G_1)$. However, since $\ZI=\ran(i\gw_k-G_1)\oplus \ran(G_2)$, in the decomposition $z=(i\gw_k-G_1)z_1+G_2y$ we must then  necessarily have $G_2y=0$, which further implies $y=0$ due to  $\ker(G_2)=\set{0}$.
Since $(z,x)\in \ker(i\gw_k- \mc{G}_1)^{n_k-1}$ was arbitrary, we have that~\eqref{eq:Gconds3} is satisfied.

Since the controller satisfies the \Gconds\ and the closed-loop system is exponentially stable, we have from Theorem~\ref{thm:RORPchar} that the controller solves the robust output regulation problem.
\end{proof}

Finally, we consider the situation where $Y$ is infinite-dimensional and the matrix $S$ in the exosystem is diagonal.
We choose
$\ZI = Y^q$ and
 the internal model in the controller is of the form
\eq{
G_1 = \diag(i\gw_1 I_Y,\ldots, i\gw_q I_Y) 
\in \Lin(\ZI),
\qquad 
G_2 = (G_2^k)_{k=1}^q
\in \Lin(Y,\ZI)
}
where the components $G_2^k$ are chosen to be boundedly invertible for all $k\in\List{q}$.

\begin{theorem}
  \label{thm:ContrTwoDiagExo}
  Assume $S=\diag(i\gw_1,\ldots,i\gw_q)$ and $P(i\gw_k)\in \Lin(U,Y)$ are boundedly invertible for all $k\in \List{q}$. If the other parameters of the controller are chosen as in the beginning of Section~\textup{\ref{sec:ContrTwo}} and if we choose
  \eq{
  K_1 = (-(G_2^1P_K(i\gw_1))^\ast, \ldots,-(G_2^qP_K(i\gw_q))^\ast)\in \Lin(\ZI,U),
  }
then the controller solves the robust output regulation problem.

If
$G_2
=  ( (I-\KtoL R(i\gw_k,A_{-1})B)P(i\gw_k)\inv)_{k=1}^q$,
then
$K_1 = (-I_Y,\ldots,-I_Y)$.
\end{theorem}

\begin{proof}
  To show that the controller solves the robust output regulation problem, it is sufficient to show that the closed-loop system is exponentially stable, because the property that the controller satisfies the \Gconds\ and all the other properties considered in the proof of Theorem~\ref{thm:ContrTwoMain} remain valid for a general Hilbert space $Y$.
  
 Since $G_1= \diag (i\gw_k I_Y)_{k=1}^q$, the operator $B_1$ is of the form $B_1=(B_1^k)_{k=1}^q\in \Lin(U,Y^q)$, where for all $k\in \List{q}$ we have
  \eq{
  B_1^k 
  = H_k B +G_2^k D
  =G_2^k P_K(i\gw_k).
  }
 since $H_k=G_2^k(\CL + D\KtoL )R(i\gw_k,A_{-1}+B\KtoL )$.
  This shows that $K_1=-B_1^\ast$.
The last claim of the theorem follows from
$P_K(i\gw_k)=P(i\gw_k)(I-\KtoL R(i\gw_k,A_{-1})B)\inv$,
 and the invertibility of $G_2^k$ and $P(i\gw_k)$ imply that $B_1^k$ are boundedly invertible for all $k\in\List{q}$. We thus have from Lemma~\ref{lem:G1fbstab} that the operator $G_1+B_1K_1=G_1-B_1B_1^\ast$ generates an exponentially stable semigroup. The exponential stability of the closed-loop system can now be shown as in the proof of Theorem~\ref{thm:ContrTwoMain}.
\end{proof}

\section{Robust Control of a 2D Heat Equation}
\label{sec:heatex}

In this section we consider robust output regulation for a two-dimensional heat equation with boundary control and observation. Set-point regulation without the robustness requirement was considered for the same system  in~\cite[Ex. VI.2]{NatGil14}.  

We study the heat equation
\eq{
x_t(\xi,t) = \Delta x(\xi,t),  \qquad x(\xi,0)=x_0(\xi) 
}
on the unit square $\xi=(\xi_1,\xi_2)\in \Omega = [0,1]\times [0,1]$. The control and observation are located 
on the parts $\Gamma_1$ and $\Gamma_2$ of the boundary $\partial \Omega$, where $\Gamma_1 = \setm{\xi=(\xi_1,0)}{0\leq \xi_1\leq 1/2}$ and $\Gamma_2 = \setm{\xi=(\xi_1,1)}{1/2\leq \xi_1\leq 1}$.
We denote $\Gamma_0=\partial \Omega\setminus (\Gamma_1\cup \Gamma_2)$. The boundary control
and the additional boundary conditions are defined as
\eq{
\pd{x}{n}(\xi,t)\vert_{\Gamma_1} = u_1(t), 
~\;
\pd{x}{n}(\xi,t)\vert_{\Gamma_2} = u_2(t), 
~\; 
\pd{x}{n}(\xi,t)\vert_{\Gamma_0} = 0 
}
for $u(t)=(u_1(t),u_2(t))\in U=\C^2$.
The outputs $y(t)=(y_1(t),y_2(t))\in Y=\C^2$ of the system are defined as averages of the value of $x(\xi,t)$ over the parts $\Gamma_1$ and $\Gamma_2$ of the boundary, i.e.,
\eq{
y_1(t)=2\int_0^{1/2}x(\xi_1,0;t)d\xi_1,  \quad 
y_2(t)=2\int_{1/2}^1x(\xi_1,1;t)d\xi_1.
}
We define $A_0=\Delta$ with domain $\Dom(A_0)=\setm{x\in H^2(\Omega)}{\pd{x}{n}=0 ~ \mbox{on} ~ \partial \Omega}$.
We have from~\cite[Cor. 1]{ByrGil02} that with the above control and observation, the heat equation is a regular linear system $(A_0,B,C,D)$ with $D=0$. The system becomes exponentially stable with negative output feedback, $u=-\kappa Cx+ \tilde{u}$
where $\kappa>0$ (cf.~\cite[Ex. VI.2]{NatGil14}). 
We choose $\kappa=1$, and define $A=( (A_0)_{-1}-BC)\vert_X$. 

Our aim is to design a minimal order controller for the stabilized system $(A,B,C,0)$ to achieve robust output tracking of the reference signal $\yref(t) = (-1,\cos(\pi t))$. To this end, we choose 
the exosystem as $W=\C^3$, $S=\diag(-i\pi,0,i\pi)$, $E=0$, and $F=-\left( {0 \atop 1/2} {-1\atop 0 } {0 \atop 1/2} \right)$. The reference signal $\yref(t)$ is then generated with the choice $v_0=(1,1,1)$ of the initial state of the exosystem.

Since $p=\dim Y = 2$,
the internal model and the parameters of the minimal order controller are given by
\eq{
\mc{G}_1 = \diag(-i\pi,-i\pi,0,0,i\pi,i\pi)\in \C^{6\times 6}, 
\qquad 
K =\eps K_0 = \eps \left( K_0^1,K_0^2,K_0^3 \right)\in \C^{2\times 6},
}
where $\eps>0$ and $K_0^k$ are to be chosen in such a way that the matrices $P(-i\pi)K_0^1$, $P(0)K_0^2$ and $P(i\pi)K_0^3$ are nonsingular. We choose $K_0^1=P(-i\pi)\inv$, $K_0^2=P(0)\inv$, and $K_0^3=P(i\pi)\inv$. Finally, 
$\mc{G}_2=(\mc{G}_2^k)_{k=1}^3$ where $\mc{G}_2^k=-I_{2\times 2}$ for $k\in \List{3}$.
We have from Theorem~\ref{thm:SimpleContrMain} that for small values of $\eps>0$ the controller achieves asymptotic tracking of the reference signal $\yref(\cdot)$, and the control structure is robust with respect to perturbations in $(A,B,C,0)$ that preserve the property $\set{0,\pm i\pi}\subset \rho(\tilde{A}))$ and the exponential stability of the closed-loop system.
In particular, this
includes small bounded perturbations to the operators $A$, $B$, $C$, and $D=0$.

The robust controller also tolerates small perturbations and inaccuracies in the parameters $K$ and $G_2$ of the controller (although robustness with respect to these operators is not required in the statement of the robust output regulation problem). Because of this property, we can use approximations for the values $P(\pm i\pi)\inv$ and $P(0)\inv$ in $K_0$. In this example we use 
a truncated eigenfunction expansion of $A_0$ in approximating the matrices $P(0)$ and $P(\pm i\pi)$.
Finally, the parameter $\eps>0$ needs to be chosen in such a way that the closed-loop is stable.

The solution of the controlled heat equation can be approximated numerically using the truncated eigenfunction expansion of the operator $A_0$. 
For the simulation, the parameter $\eps$ is chosen to be $\eps=1/4$.  
Figure~\ref{fig:heatsimy1y2} depicts the simulated behaviour of the two outputs of the plant. The solution of the controlled partial differential equation at time $t=16$ is plotted in Figure~\ref{fig:heatsimsol}.

\begin{figure}[h!]
  \begin{minipage}[b]{0.47\linewidth}
    \begin{center} \includegraphics[width=\linewidth]{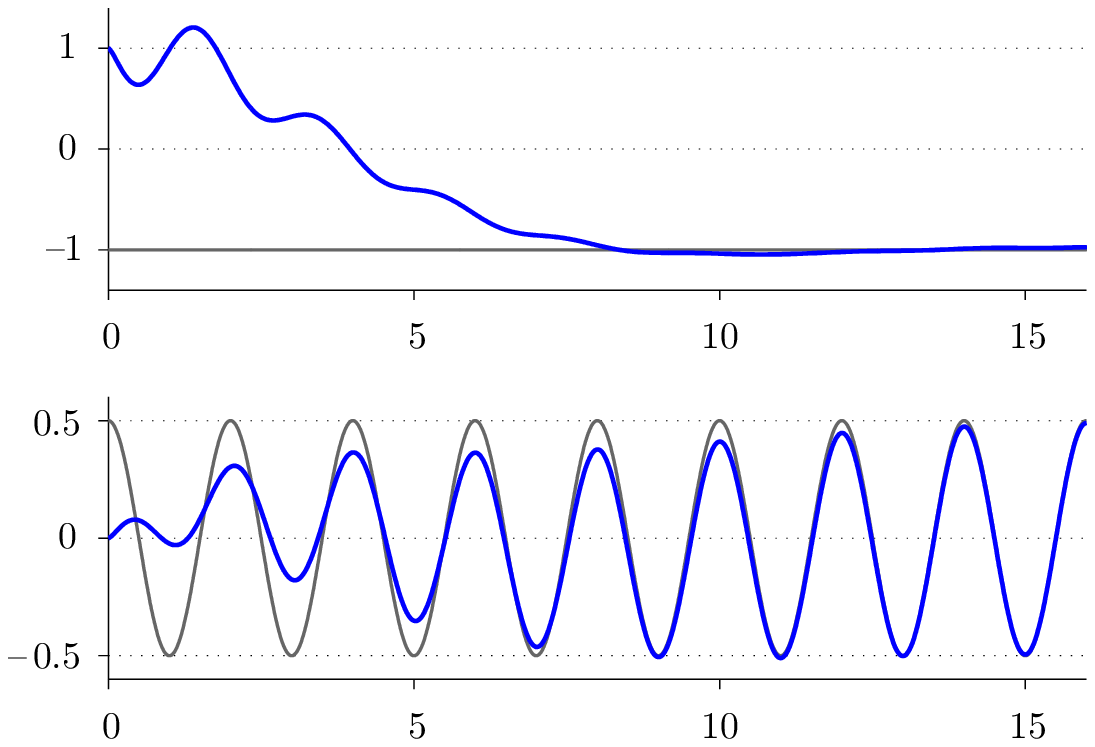}
    \end{center}
    \caption{Outputs $y_1(\cdot)$ and $y_2(\cdot)$ of the controlled system.}
    \label{fig:heatsimy1y2}
  \end{minipage}
\hfill
  \begin{minipage}[b]{0.47\linewidth}
  \begin{center}
\includegraphics[width=\linewidth]{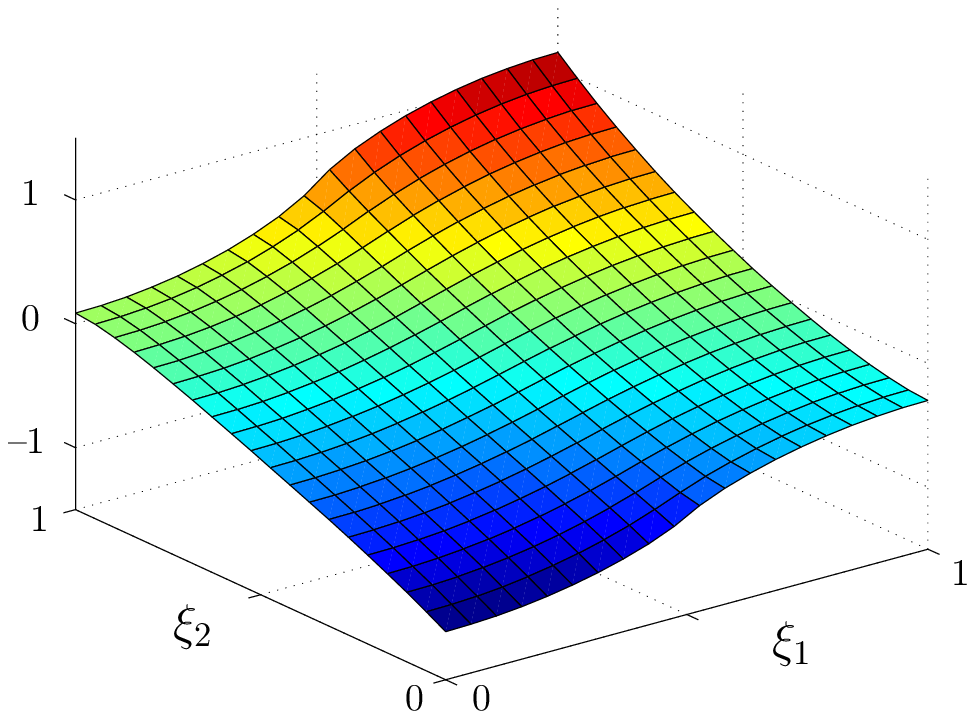}
  \end{center}
  \caption{State of the controlled system at time $t=16$.}
  \label{fig:heatsimsol}
\end{minipage}
\end{figure}

\appendix

\section{Appendix}

\begin{lemma}
  \label{lem:G1fbstab}
  Let $G_1 = \diag(i\gw_1 I_Y,\ldots,i\gw_q I_Y)\in \Lin(Y^q)$ and $G_2 = (G_2^k)_{k=1}^q \in \Lin(U,Y^q)$ where $U$ and $Y$ are Hilbert spaces. If the components $G_2^k$ of $G_2$ are boundedly invertible for all $k\in \List{q}$, then the semigroup generated by $G_1-G_2G_2^\ast$ is exponentially stable.
\end{lemma}

\begin{proof}
  Since $G_1-G_2G_2^\ast$ is a bounded operator, it is sufficient to show that $\gs(G_1-G_2G_2^\ast)\subset \C^-$.
  Since $G_1$ generates a contraction semigroup, the same is true for $G_1-G_2 G_2^\ast$, and thus $\gs(G_1-G_2G_2^\ast)\subset \overline{\C^-}$. 
It therefore remains to show that $i\R\subset \rho(G_1-G_2G_2^\ast)$.

Let $i\gw\in i\R$ be such that $\gw\neq \gw_k$ for all $k\in\List{q}$. We then have $i\gw\in \rho(G_1)$.
If $I+G_2^\ast R(i\gw,G_1)G_2 $ is boundedly invertible, then 
the Woodbury formula implies that 
  $i\gw\in\rho(G_1+ G_2 G_2^\ast)$.
However, this
is true since
$ G_2^\ast R(i\gw, G_1)G_2$ 
is bounded and skew-adjoint.

It remains to consider the case where $i\gw=i\gw_n$ for some 
$n\in\List{q}$. 
We will show 
 $\norm{(i\gw_n-G_1 +  G_2G_2^\ast )z }\geq c \norm{z}$ for some constant $c>0$ and for all $z\in Z$. 
If this is not true, there exists a sequence $(z_k)_{k\in\N}\subset Z$ such that $\norm{z_k}=1$ for all $k\in \N$ and $\norm{(i\gw_n-G_1 +  G_2G_2^\ast)z_k }\to 0$ as $k\to \infty$.
Since $i\gw_n-G_1$ is skew-adjoint, we have
\eq{
\MoveEqLeft \norm{(i\gw_n-G_1+G_2G_2^\ast )z_k}
\geq\abs{\iprod{(i\gw_n-G_1+G_2G_2^\ast )z_k}{z_k}}\\
&\geq \abs{\re\iprod{(i\gw_n-G_1+G_2G_2^\ast )z_k}{z_k})}
= \norm{G_2^\ast z_k}^2,
}
and thus $\norm{G_2^\ast z_k} \to 0$ as $k\to \infty$. 
For every $k\in\N$ denote
$z_k=z_k^1+z_k^2$ where $z_k^1\in \ran(i\gw_n-G_1)$, $z_k^2\in\ker(i\gw_n-G_1)$, and $1=\norm{z_k}^2=\norm{z_k^1}^2+\norm{z_k^2}^2$.
There exists $c_1>0$ such that
$\norm{(i\gw_n-G_1)z_k^1}\geq c_1 \norm{z_k^1}$ for all $k\in\N$.
Thus
\eq{
 c_1\norm{z_k^1}
 &\leq \norm{(i\gw_n-G_1)z_k^1} =
\norm{(i\gw_n-G_1)z_k}\\
&\leq \norm{(i\gw_n-G_1+G_2G_2^\ast)z_k}+ \norm{G_2}\norm{G_2^\ast z_k}
\to 0 
}
as $k\to\infty$.
Moreover, $\norm{G_2^\ast z_k^2}\geq \norm{(G_2^n)\inv}\inv \norm{z_k^2}$, and 
\eq{
\norm{(G_2^n)\inv}\inv \norm{z_k^2}\leq 
\norm{G_2^\ast z_k^2}\leq  \norm{G_2^\ast z_k}+\norm{G_2^\ast z_k^1}\to 0
}
as $k\to\infty$.
We have now shown that
 $z_k^1\to 0$ and $z_k^2\to 0$, but this contradicts the assumption that
 $ \norm{z_k}^2 =1$ for all $k\in\N$,
and thus
the original claim holds.
In particular $i\gw_n\notin \gs_p(G_1+G_2G_2^\ast)$ and the range of $i\gw_n-G_1+G_2G_2^\ast$ is closed. Finally, the Mean Ergodic Theorem~\cite[Sec. 4.3]{AreBat01book} implies that the range of $i\gw_n-G_1+G_2G_2^\ast$ is dense, and
$i\gw_n\in\rho(G_1+G_2G_2^\ast)$.
\end{proof}

\end{document}